\documentclass{article}
\usepackage[hmargin=2.5cm,vmargin=3cm,bindingoffset=0.5cm]{geometry}

\usepackage{fancyhdr,graphics,graphicx,xy}
\usepackage{ifthen}
\usepackage{amssymb,amsmath,mathtools,enumerate,tikz-cd}
\usepackage{amsthm}
\usepackage{graphics}
\usepackage{dsfont}
\usepackage{algpseudocode}
\usepackage{adjustbox}
\usepackage{bbm}
\usepackage{tikz}
\usepackage[affil-it]{authblk}
\usepackage{circuitikz}
\usepackage[permil]{overpic}

\usetikzlibrary{arrows}
\usetikzlibrary{snakes}
\graphicspath{/Users/huubdejong/Downloads/}
\def\N{\mathbb{N}}
\def\Q{\mathbb{Q}}
\def\C{\mathbb{C}}
\def\R{\mathbb{R}}

\def\Z{\mathbb{Z}}

\def\T{\mathbb{T}}

\def\M{\mathcal{M}}

\def\mS{\mathcal{S}}
\def\mJ{\mathcal{J}}

\def\mV{\mathcal{V}}

\def\limn{\lim_{n\to\infty}}
\def\limk{\lim_{k\to\infty}}

\def\diam{\operatorname{diam}}
\def\lcm{\operatorname{lcm}}
\def\topint{\operatorname{int}}

\newcommand{\brac}[1]{\left({#1}\right)}

\newcommand{\Set}[1]{\left\{ #1 \right\}}
\newcommand{\set}[1]{\{ #1 \}}

\title{Smoothness of Markov Partitions for Expanding Toral Endomorphisms}
\author{Chayce Hughes%
  \thanks{Electronic address: \texttt{chahug21@student.ubc.ca}}\;}
\author{Huub de Jong%
  \thanks{Electronic address: \texttt{huub.dejong@math.ubc.ca}}}

\affil{The University of British Columbia, Vancouver}
\date{\today}

\newtheorem{theorem}{Theorem}[section]
\newtheorem{corollary}[theorem]{Corollary}
\newtheorem{definition}[theorem]{Definition}
\newtheorem{lemma}[theorem]{Lemma}

\newtheorem{proposition}[theorem]{Proposition}

\newtheorem{observation}[theorem]{Observation}

\begin{document}

\maketitle
\begin{abstract}
    We show that an expanding toral endomorphism admits a Markov partition with a piecewise smooth boundary if and only if some power of the inducing matrix is diagonalizable and all of its eigenvalues are integers. Moreover, such a piecewise smooth boundary must in fact be piecewise linear. In dimension 2, we make some more subtle observations about the smoothness of the boundary. We also provide an upper bound on the Hausdorff dimension.
\end{abstract}

\section{Introduction} \label{Section:Intro}

Markov partitions were developed by Adler and Weiss for hyperbolic toral automorphisms \cite{AW}, and systematically studied by Sinai in the context of Anosov diffeomorphisms on manifolds \cite{Si}. The utility of Markov partitions comes from the fact that they give rise to a symbolic description of the system which retains large amounts of interesting information. A great deal of advancment in the theory on Markov partitions is due to Bowen, who used these partitions as a tool in better understanding the more general class of Axiom A diffeomorphisms \cite{Bo1}. He showed that such diffeomorphisms admit a Markov partition and used them to show several dynamical properties, such as the 0-dimensionality of minimal sets \cite{Bo3}. 

A particularly beautiful 3-page paper of Bowen \cite{Bo2} shows that no Markov partition for a hyperbolic toral automorphism in dimension 3 can have a piecewise smooth boundary. This result was later extended by Cawley \cite{C} to completely classify when a hyperbolic toral automorphism admits a smooth Markov partition, and that such a Markov partition must in fact be linear. For the precise definition of smoothness and linearity used in the statement here we refer to the original paper.

\begin{theorem}[Cawley, 1991]\label{thm:Cawley}
    A hyperbolic toral automorphism $f:\T^d\to\T^d$ induced by a matrix $A$ admits a smooth Markov partition if and only if some power of $A$ is similar over $\Q$ to a block diagonal matrix $$\begin{pmatrix}
        L_1 & & & \\
        & L_2 & & \\
        & & \ddots & \\
        & & & L_{d/2}
    \end{pmatrix}$$
    where each $L_i$ is a $2\times 2$ hyperbolic integer matrix with $\det L_i = \pm1$. 

    Moreover, any smooth Markov partition for a hyperbolic toral automorphism must be linear.
\end{theorem}

An equivalent way of phrasing the first part of this theorem is that a hyperbolic toral automorphism $f$ admits a smooth Markov partition if and only if there exists an $m>0$ such that $f^m$ is finitely, algebraically covered by a direct product of $2$-dimensional hyperbolic toral automorphisms. Indeed, call the block diagonal matrix above $L$ and suppose $TA^m=LT$ for some rational matrix $T$. We can rescale by an integer and assume $T$ has integer entries. This means that $T$ induces a well-defined $|\det T|$-to-one map on the torus $f_T:\T^d\to\T^d$, and we get a commutative diagram $$\begin{tikzcd}
\mathbb{T}^d \arrow[d, "f_T"] \arrow[r, "f_L"] & \mathbb{T}^d \arrow[d, "f_T"] \\
\mathbb{T}^d \arrow[r, "f_A^m"]                & \mathbb{T}^d                 
\end{tikzcd}$$

In general, Markov partitions are developed for invertible systems, but the theory applies in the context of non-invertible, expanding dynamical systems as well. The existence of Markov partitions in such systems is shown in e.g. Chapter 4 of \cite{URM} or Chapter 7 of \cite{Ru}, and a side-by-side development of the invertible and non-invertible settings can be found in \cite{AKS}. Work of Bedford \cite{Be} uses explicitly fractal techniques due to Dekking \cite{D} to construct Markov partitions for expanding toral endomorphisms in dimension 2. More recently, Mihailescu has given a non-constructive proof that expanding toral endomorphisms with the Lebesgue measure on $\T^d$ are one-sided Bernoulli \cite{Mi}. Under certain assumptions, Leuridan used lattice tilings to provide a construction which can roughly be interpreted as a measurable variant of Markov partitions \cite{L}.

A folkloric result says that our Markov partitions ``almost always'' have non-smooth boundary, but we were not able to locate a precise statement in the literature. This work attempts to address this apparent gap and proves some smoothness results that seem to parallel Theorem \ref{thm:Cawley}, although we obtain them by elementary tools, avoiding the heavy topological machinery used in \cite{C}. More precisely, we prove the following; the exact definition of a smooth or linear Markov partition will be given in Section \ref{Section:Prelims}.

\begin{theorem}\label{thm-A smooth MP iff A^k diagonalizable with integer eigenvals}
    An expanding toral endomorphism $f:\T^d\to \T^d$ induced by a matrix $A$ admits a smooth Markov partition if and only if some power of $A$ is similar over $\Q$ to a diagonal integer matrix $$\begin{pmatrix}
        k_1 & & & \\
        & k_2 & & \\
        & & \ddots & \\
        & & & k_{d}
    \end{pmatrix}$$
    Equivalently, some power of $A$ is diagonalizable with integer eigenvalues.
\end{theorem}

In other words, an expanding toral endomorphism $f$ admits a smooth Markov partition if and only if there exists an $m>0$ such that $f^m$ is finitely, algebraically covered by a direct product of $1$-dimensional expanding toral endomorphisms.

We will also use Theorem \ref{thm-A smooth MP iff A^k diagonalizable with integer eigenvals} to prove the next result, which is the same conclusion as the second half of Theorem \ref{thm:Cawley}.

\begin{theorem}\label{thm-smooth is linear}
    Let $A$ be a $d\times d$ expanding integer matrix. Then, any smooth Markov partition for $A$ is linear.
\end{theorem}

The structure of this paper is as follows. Section \ref{Section:Prelims} provides the background on Markov partitions and their properties. Section \ref{Section: Tilings} will prove the ``if'' direction of Theorem \ref{thm-A smooth MP iff A^k diagonalizable with integer eigenvals} and relate Markov partitions for matrices that are up to a power similar over $\Q$. Section \ref{Section: Non-smoothness} is dedicated to the ``only if'' direction of Theorem \ref{thm-A smooth MP iff A^k diagonalizable with integer eigenvals}. In Section \ref{Section: linearity}, we discuss smooth Markov partitions in further detail, proving Theorem \ref{thm-smooth is linear} and analyzing the hypotheses of Theorem \ref{thm-A smooth MP iff A^k diagonalizable with integer eigenvals}. In Section \ref{Section: Dimension 2}, we expand on the notion of non-smoothness in dimension 2, and in Section \ref{Section: Hausdorff dim} we provide an upper bound for the Hausdorff dimension of the boundary of any Markov partition.

\section*{Acknowledgements}

The authors gratefully acknowledge that this paper was inspired by the work of Elise Cawley \cite{C}. We thank Mariusz Urbánski for providing background, advice, and encouragement. We also give thanks to Brian Marcus, Chengyu Wu, Alexia Yavicoli, and Jim Bryan for stimulating discussion and detailed feedback. 

\section*{Declaration on the use of generative AI}

No generative AI was used to produce the results, proofs, and writing in this paper. While the authors have used ChatGPT in the way one would use the Google search engine earlier on in this project, no meaningful findings were incorporated into the paper, and our stances on generative AI have shifted significantly since then. Responsibility for any mathematical or typographical errors is entirely ours.

\section{Preliminaries} \label{Section:Prelims}

A \textit{dynamical system} is a compact metric space $X$ with a continuous self-map $f:X\to X$. The map $f:X\to X$ is \textit{positively expansive} if there exists some $c>0$, called an \textit{expansive constant},  such that for all $x,y\in X$, $x\neq y$ implies there exists an $n\in \N=\{0,1,2,...\}$ such that $d(f^nx,f^ny)>c$. We are interested in \textit{expanding toral endomorphisms} on the $d$-dimensional torus $\T^d:= \R^d/\Z^d$, which are maps induced from a $d\times d$ integer matrix $A$ for which all eigenvalues have modulus strictly bigger than 1. The induced expanding toral endomorphism is $$f_A:\T^d\to\T^d$$ defined as $$f_A(x+\Z^d)=Ax+\Z^d.$$ It is not hard to show that the condition on the modulus of the eigenvalues implies that $f_A$ is positively expansive. We will call $A$ an \textit{expanding} integer matrix and write $f$ for $f_A$ when no ambiguity exists. 

A well-studied and useful property of expanding toral endomorphisms is that the set of periodic points is dense. In particular, $f_A$ acts by permutation on the discrete subgroups $$\left\{\left(\frac{k_1}{n},...,\frac{k_d}{n}\right)\mod 1:k_i\in\Z\right\}$$ whenever $\gcd(n,\det A)=1$, see e.g. \cite{LN}. 

For two integer matrices $A,B$, we write $A\sim B$ to mean that $A$ and $B$ are similar over $\Q$, i.e. there exists an invertible matrix $T$ with rational entries such that $A=TBT^{-1}$. We also write $\rho:\R^d\to\T^d$ for the standard projection map, and consider the standard metric on $\T^d$, $$d_\T(x,y):=\min_{\substack{u,v\in \R^d \\ \rho(u)=x \\\rho(v)=y}} d_\R(u,v)$$ where $d_\R(u,v)$ is the Euclidian metric on $\R^d$. 

We will employ several notational conveniences when using metrics and the projection map. When there is no ambiguity, we may omit the subscripts from the notation. Moreover, we will let $\rho(\cdot)$ and $d_\T(\cdot,\cdot)$ be the projection map and distance metric for any dimension. This will make our notation less cumbersome, since we will be dealing with tori of different dimensions simultaneously. The maps should always be interpreted as functions on the spaces with the dimension determined by the input. For example, if $(x,y)\in \R^2$, $\rho((x,y))$ will be an element of $\T^2$ and $\rho (x)$ will be an element of $\T^1$. Under this convention, we will also abbreviate $$d_\T(\rho(\cdot),\rho(\cdot))$$ to just $d_\T(\cdot,\cdot)$. That is, if $x,y\in \R^d$, then $d_\T(x,y)$ means $d_\T(\rho(x),\rho(y))$.   

Since $f:\T^d\to \T^d$ is a positively expansive map from a manifold onto itself, a theorem of Coven and Reddy \cite{CR} gives a metric $\tilde{d}$ compatible with the standard metric for which there exist constants $\varepsilon>0,\lambda >1$, so that \begin{equation*}
    \tilde d(x,y)<\varepsilon \implies \tilde d(fx,fy)>\lambda \tilde d(x,y)
\end{equation*} 
and \begin{equation*}
    \tilde d(fx,z)<\varepsilon\lambda \implies B(x,\varepsilon)\cap f^{-1}(z) \text{ is a singleton.}
\end{equation*} That is, $f$ is \textit{expanding} in their terminology with respect to $\tilde{d}$, which need not be equal to $d$. In Section \ref{Section: Hausdorff dim}, we will give more explicit expressions for this metric.

A set $R\subset X$ of a topological space is \textit{proper} if $R=\overline{\operatorname{int}(R)}\neq \emptyset$. Following \cite{AKS}, a \textit{Markov partition} for an expansive dynamical system $(X,f)$ is a finite cover $\M=\{R_i\}_{i=1}^n$ of $X$ by proper sets with disjoint interiors such that two things hold: \begin{enumerate}
    \item $\M$ is \textit{metrically generating}, meaning for all $(x_i)\in \{1,...,n\}^\N$, $$\bigcap_{k=0}^\infty f^{-k}(\operatorname{int}(R_{x_k}))$$ consists of at most one point. 
    \item $\M$ satisfies a \textit{Markov property} (MP), meaning that for all $R,R'\in \M$, if $$f(\operatorname{int}(R))\cap \operatorname{int} (R')\neq \emptyset,$$ then each $y\in \operatorname{int} (R')$ has a unique preimage in $R$. 
\end{enumerate}

For constructive arguments, it is convenient to use a slightly stronger version of the Markov property above: 

\begin{enumerate}[(MP+)]
    \item For all $R\in \M$, $f|_R$ is injective, and there exist $k=k(R)\in\{1,...,n\}$ and $R_1,...,R_k\in \M$ such that $$f(R)=\bigcup_{i=1}^kR_i.$$
\end{enumerate} The property (MP+) directly implies (MP), but it also captures more explicitly the intuition behind Markov partitions for non-invertible systems. 

The reason we insist on the weaker property (MP) is that we want to include Markov partitions where the sets $R_i$ have ``large'' diameter. A standard and very intuitive example of a Markov partition is to consider the expanding circle map $f:\T^1\to\T^1$ given by $f(x)=2x\mod1$ and the sets $[0,1/2],[1/2,1]$. This Markov partition is exactly what is used when writing elements of the unit interval in their binary expansion, but it is clear that (MP+) is not satisfied. 

At the same time, (MP) is in practice not significantly weaker than (MP+). Any positively expansive map is injective on small enough sets, and in Lemma \ref{lem-backwards refinements of MP} we show that any Markov partition can be transformed into an arbitrarily small one without changing the geometry of its boundary. In particular, this new Markov partition will also satisfy (MP+).

Existence of Markov partitions is commonly shown for systems $(X,f)$ which are expanding with respect to a given metric, which is why we point out the existence of $\tilde d$ and the expanding property. However, the Markov partition itself is a purely topological object and makes no reference to the metric. For details on Markov partitions for expanding systems in generality, we refer to Chapter 4 of \cite{URM} or Chapter 7 of \cite{Ru}.

We denote the \textit{boundary} of a Markov partition $$\partial\M:=\bigcup_{R\in \M} \partial R.$$ Moreover, we call a set $X$ \textit{invariant} under a map $f$ if $f(X)\subset X$. It follows from the definition that the boundary of a Markov partition is a nowhere dense invariant set.

\begin{lemma}\label{lem-boundary is invariant and nowhere dense}
    Let $\mathcal{M}$ be a Markov partition for an expanding toral endomorphism $f:\T^d\to\T^d$. Then $f(\partial\M)\subset \partial\M$ and $\partial\M$ is nowhere dense in $\T^d$. 
\end{lemma}

\begin{proof}
    Let $x\in \partial R$ and suppose $f(x)\in \operatorname{int}(R')$ for some $R,R'\in \M$. Moreover, let $c>0$ be an expansive constant for $f$. For all $n\in \N$ sufficiently large, take $\delta_n\in(0,c/2)$ so that $$f(B(x,\delta_n))\subset B\left(f(x),\frac{1}{n}\right) \subset \operatorname{int}(R').$$
    Since $x\in \partial R$, there exists some $\tilde{x}_n\in B(x,\delta_n)\setminus R$, and also $f(\topint R)\cap \topint R'\neq \emptyset$. By the Markov property, $f(\tilde{x}_n)\in \operatorname{int}(R')$ has a unique preimage $z_n$ in $R$, but $z_n$ cannot be in $B(x,c/2)$, for otherwise $z_n=\tilde{x}_n$ by positive expansiveness. Take now a convergent subsequence $z_{n_k}$ with $\bar{z}=\limk z_{n_k}\in R$, since $R$ is closed. Then $f(\bar{z})=\limk f(z_{n_k}) = f(x)$, and $\bar{z}\neq x$ since $d(x,\bar{z})\geq c/2$. But this contradicts the Markov property, so we are done. 

    That $\partial\M$ is nowhere dense follows immediately from the definition. 
\end{proof}

We now define the smoothness property for a Markov partition. 

\begin{definition}
    A Markov partition $\M$ for an expanding toral endomorphism $f:\T^d\to\T^d$ is smooth if we can write $\partial\M=\bigcup_{i=1}^n\mathcal{S}_i$ where each $\mathcal{S}_i$ is a $C^1$ hypersurface, i.e. an embedded submanifold of dimension $d-1$, possibly with boundary. If each $\mathcal{S}_i$ is a linear hypersurface, we say $\M$ is linear. 
\end{definition}

The reason Markov partitions are of great interest in dynamical systems is that they allow for the construction of a \textit{symbolic coding} of a system. We briefly introduce this coding here, focusing on the aspects that are relevant for our results; for more details we refer to \cite{AKS} or Chapter 6 of \cite{LM}. 

A given Markov partition $\M$ for a dynamical system $(X,f)$ defines an $n\times n$ matrix $T$ where $$T(i,j)=\begin{cases}
    1,& \text{if } f(\operatorname{int}(R_i))\cap \operatorname{int}(R_j) \neq \emptyset \\
    0,& \text{otherwise}
\end{cases}.$$ Such a matrix $T$ in turn defines a \textit{shift of finite type} (SFT) 

$$\Sigma_T:=\{(x_0x_1x_2...)\in \{1,...,n\}^\N:\forall i\in\N, T(x_i,x_{i+1})=1\},$$ which we equip with the shift map $\sigma$ defined by $$\sigma(x)_i=x_{i+1}.$$ We give $\{1,...,n\}$ the discrete topology and $\{1,...,n\}^\N$ the corresponding product topology. Then $\Sigma_T$ is compact in the subspace topology, $\sigma$-invariant, and metrizable. 

Observation 1.4 of \cite{AKS} shows that a Markov partition yields a continuous and surjective map $\pi:\Sigma_T\to X$ such that the diagram
$$\begin{tikzcd}
\Sigma_T \arrow[r, "\sigma"] \arrow[d, "\pi"] & \Sigma_T \arrow[d, "\pi"] \\
X \arrow[r, "f"]                              & X                        
\end{tikzcd}$$ commutes. The map $\pi$ described there, although not presented in this form, is defined as $$\pi(x):=\bigcap_{N=0}^\infty \overline{\bigcap_{i=0}^Nf^{-i}\operatorname{int}(R_{x_i})},$$ where we identify the intersection with the (necessarily unique) point it contains. The map $\pi$ is called a \textit{coding} map.

The results by Bowen and Cawley presented in Section \ref{Section:Intro} rely on the (non-)existence of invariant sets with certain regularity properties. We will make similar arguments, but phrase them in terms of the eigenvalues and characteristic polynomial of the inducing matrix. The following lemma holds for any integer matrix, not just expanding ones. While the result is standard, we include the proof for completeness.

\begin{lemma}\label{lem-invariant subtori}
    Let $A$ be an integer matrix and $f_A:\T^d\to\T^d$ the induced toral map. Then $f_A$ has a non-trivial invariant subtorus if and only if $\chi_A$ is reducible. 
\end{lemma}

\begin{proof}
    We first assume that $\chi_A$ is irreducible. Let $S\subseteq \T^d$ be an invariant subgroup homeomorphic and isomorphic to $\T^k$ with $k>0$. Then $\rho^{-1}(S)$ is a closed subgroup of $\R^d$ and invariant under the action of $A$. Using e.g. Chapter VII.1, Theorem 2 of \cite{Bou}, $$\rho^{-1}(S)\cong \Lambda\oplus V$$ where $\Lambda$ is a discrete subgroup of $\R^d$ and $V$ is a linear subspace. Since $S\neq \{0\}$, $V$ has dimension at least 1, and since $A$ fixes the origin and preserves path-connected components, $AV\subseteq V$. 

    Now, compactness of $S$ implies that there exists a $\vec{b}\in V\cap \Z^d\setminus\{0\}$, for otherwise $\rho|_V$ would be injective. We assumed that $\chi_A$ is irreducible, which implies that for any set of $k<d$ vectors $\{\vec{b}_1,...,\vec{b}_k\}\subset \Z^d$, $$A\operatorname{span}(\vec{b}_1,...,\vec{b}_k)\not\subseteq \operatorname{span}(\vec{b}_1,...,\vec{b}_k).$$ If such a set were invariant, we could complete $\{\vec{b}_1,...,\vec{b}_k\}$ to a basis of $\R^d$ of vectors in $\Z^d$ with change of basis matrix $T$ giving $$A=T\begin{pmatrix}
        B & D \\0 & C
    \end{pmatrix}T^{-1},$$ contradicting irreducibility of $\chi_A$.
    
    In particular, $$A\vec{b}\notin \operatorname{span}(\vec{b}),$$  $$A^2b\notin \operatorname{span}(\vec{b},A\vec{b}),$$  $$A^3\vec{b}\notin \operatorname{span}(A^2\vec{b},A\vec{b},\vec{b})$$ and so on. Thus, $V$ contains $d$ linearly independent vectors $$\vec{b},A\vec{b},...,A^{d-1}\vec{b},$$ so $V=\R^d$ and $S=\T^d$. 

For the converse, if $\chi_A$ is reducible, by Theorem III.12 of \cite{N}, $$A\sim \begin{pmatrix}
    B & D \\0 & C
\end{pmatrix},$$ which immediately gives an invariant subspace $E_B$. Then $\rho(E_B)$ is a non-trivial invariant subtorus. 
\end{proof}

The specific connection to invariant subtori for our purposes will use a result of Mañé \cite{Ma} concerning the orbit of a certain curve under iterates of an expanding toral endomorphism in combination with Lemma \ref{lem-boundary is invariant and nowhere dense}. In \cite{Ma}, the main result is stated differently and for hyperbolic toral automorphisms. However, using Lemma \ref{lem-invariant subtori} and carefully analyzing the proof and Remark II given by Mañé allows us to rephrase the result as below. To illustrate the relation to our main result, it is immediate that any matrix of dimension bigger than 1 satisfying the assumptions below cannot admit a smooth Markov partition.

\begin{lemma}[Mañé, 1979]\label{lem-Mane}
    Let $A$ be an expanding integer matrix such that for all $k\in \N\setminus\{0\}$, $\chi_{A^k}$ is irreducible. Then for every $\varepsilon>0$ and any rectifiable, non-constant curve $\gamma:[0,1]\to \T^d$, there exists an $n\in \N$ such that $f_A^n\gamma$ is $\varepsilon$-dense in $\T^d$.
\end{lemma}

Here, $\varepsilon$-density of a set $S$ in $\T^d$ means that for every $x\in \T^d$, there exists an $s\in S$ with $d_\T(s,x)<\varepsilon$. Moreover, we use $\gamma$ to refer to both the parametrization of a curve and the set of points $\gamma(|[0,1])$.

\section{Tilings and construction of linear Markov partitions}\label{Section: Tilings}

The goal of this section is to first prove the following.

\begin{proposition}\label{prop-reduction to similarity and powers}
    Let $A,B$ be $d\times d$ expanding integer matrices such that $A^m=QBQ^{-1}$ for some $Q\in \operatorname{Gl}_d(\Q)$ and some $m\in \N\setminus \{0\}$. Then $A$ admits a smooth (respectively linear) Markov partition if and only if $B$ does.
\end{proposition}

In order to do so, it will be convenient to introduce some notation and terminology. For a topological space $X$ we call a collection $\mathfrak{C}=\{C_i\}_{i\in I}$ a \textit{tiling} if the sets $C_i$ are compact proper subsets with disjoint interiors whose union covers $X$. To avoid trivialities, we also ask that a tiling consists of more than 1 element. Markov partitions are a tiling in this sense (except for the trivial system consisting of a single fixed point), and we carry over the notation $$\partial\mathfrak{C}=\bigcup_{i\in I} \partial C_i.$$ Given a finite collection of tilings $\mathfrak{C}^1,...,\mathfrak{C}^n$, write $$\bigvee_{i=1}^n\mathfrak{C}^i:= \left\{\overline{\bigcap_{i=1}^n\topint C^i}:C^i\in \mathfrak{C}^i\right\},$$ where we omit empty sets. 

The following lemma is an exercise in point-set topology. The proof is included for completeness, but may readily be skipped by the reader.

\begin{lemma}\label{lem-intersecting tilings gives tilings}

    Let $\mathfrak{C}^1,...,\mathfrak{C}^n$ be tilings of a topological space $X$. Then $\bigvee_{i=1}^n\mathfrak{C}^i$ is a tiling of $X$ where $$\topint \left(\overline{\bigcap_{i=1}^n\topint C^i}\right)={\bigcap_{i=1}^n\topint C^i}$$ and $$\partial \left(\bigvee_{i=1}^n\mathfrak{C}^i\right)=\bigcup_{i=1}^n \partial\mathfrak{C}^i.$$
\end{lemma}

\begin{proof}
    We first show that $$\topint \left(\overline{\bigcap_{i=1}^n\topint C^i}\right)=\bigcap_{i=1}^n\topint C^i.$$

    The inclusion $\supseteq$ is clear, so we suppose towards a contradiction that $\subseteq$ fails, i.e. there exists an open set $V\subseteq \overline{\bigcap_{i=1}^n\topint C^i}$ with a $z\in V$ such that $z\notin {\bigcap_{i=1}^n\topint C^i}$. Hence, there is an $i\in \{1,...,n\}$ for which $z\notin \topint C^i$. Since $$\overline{\bigcap_{i=1}^n\topint C^i}\subseteq \bigcup_{i=1}^n\overline{\topint C^i},$$ we have $$z\in \overline{\topint C^i}\setminus \topint C^i.$$ Since $C^i$ is closed and not all of $X$, this means $V\cap (C^i)^c$ is a nonempty open set, but this contradicts $V\subseteq\overline{\bigcap_{i=1}^n\topint C^i}$. It is now immediate that the sets $\overline{\bigcap_{i=1}^n\topint C^i}$ are proper compact with disjoint interiors and cover $X$.

    Now, note that for any proper set $S$, $\partial S=\partial(\topint S)$, so \begin{align*}
        \partial\left(\overline{\bigcap_{i=1}^n\topint C^i}\right)= \partial\left({\bigcap_{i=1}^n\topint C^i}\right) \subseteq \bigcup_{i=1}^n \partial(\topint C^i) = \bigcup_{i=1}^n\partial C^i.
    \end{align*}
    Hence, $$\partial \left(\bigvee_{i=1}^n\mathfrak{C}^i\right)\subseteq \bigcup_{i=1}^n \partial\mathfrak{C}^i.$$

    Conversely, fix $C\in \mathfrak{C}^1$. Then $$\topint C \supseteq \bigcup_{C^i\in\mathfrak{C}^i, i\neq 1} \topint C \cap \topint C^2 \cap...\cap\topint C^n$$ and $$C=\overline{\bigcup_{C^i\in\mathfrak{C}^i, i\neq 1} \topint C \cap \topint C^2 \cap...\cap\topint C^n},$$ so \begin{align*}\partial C&\subseteq \partial\left(\overline{\bigcup_{C^i\in\mathfrak{C}^i, i\neq 1} \topint C \cap \topint C^2 \cap...\cap\topint C^n}\right) \\
    &\subseteq \bigcup_{C^i\in\mathfrak{C}^i, i\neq 1} \partial\left(\overline{\topint C \cap \topint C^2 \cap...\cap\topint C^n}\right)\end{align*}
    Repeating the argument for $i=2,...,n$, we have shown that $$\bigcup_{i=1}^n \partial\mathfrak{C}^n\subseteq \partial\left(\bigvee_{i=1}^n\mathfrak{C}^i\right),$$ and we are done.
\end{proof}

We will also a technique known in symbolic dynamics as \textit{recoding to a higher block code} (see Section 1.4 of \cite{LM}), but applied at the level of Markov partitions. This will allow us to make a given Markov partition arbitrarily small, without changing the geometry of the boundary. Let $f:\T^d\to \T^d$ be an expanding toral endomorphism and $\M=\{R_i\}_{i=1}^n$ a Markov partition for $f$. For a fixed $N\in\N$ and $x_0...x_N\in\{1,...,n\}^{N+1}$, write $$R_{x_0...x_N}:=\overline{\bigcap_{i=0}^Nf^{-i}\left(\operatorname{int}R_{x_i}\right)}$$ and $\M^N:=\{R_{x_0...x_N}\}$ for the collection of all such non-empty sets. We call $\M^N$ a \textit{refinement} of $\M$, since by Lemma 3.2 of \cite{AKS}, $$\lim_{N\to\infty} \sup\{\diam(R):R\in \mathcal{M}^N\}=0.$$ That is, we can make the elements of $\M^N$ as small as we like.

\begin{lemma}\label{lem-backwards refinements of MP}
    Let $f:\T^d\to \T^d$ be an expanding toral endomorphism with Markov partition $\M$. Then in the notation as above, $\M^N$ is a Markov partition for $f$ and $\partial\M^N=f^{-N}\partial\M$. 
\end{lemma}

\begin{proof}
    Note that an expanding toral endomorphism $f$ is a local homeomorphism, so $f^{-1}$ commutes with taking closures, interiors, and boundaries. Thus, for any $k\in \N$, and $R\in \M$, $f^{-k}(R)$ is a proper set, and so $\{f^{-k}(R_i)\}_{i=1}^n$ covers $\T^d$ by proper compact sets with disjoint interiors. Now by applying Lemma \ref{lem-intersecting tilings gives tilings} the sets $R_{x_0...x_N}$ are a cover of $\T^d$ by proper compact sets with disjoint interiors. 

    To check that these sets are metrically generating, let $$((x_{(k,0)},...,x_{(k,N)}))_{k\in \N}\in \left(\{1,...,n\}^{N+1}\right)^\N.$$ Then \begin{align*}
        \bigcap_{k=0}^\infty f^{-k}\topint\left(\overline{\bigcap_{i=0}^Nf^{-i}\left(\operatorname{int}R_{x_{(k,i)}}\right)}\right) \subseteq \bigcap_{k=0}^\infty f^{-k} \topint R_{x_{(k,0)}}
    \end{align*}   
    contains at most a single element because $\M$ is metrically generating.

    We now show the Markov property. Suppose that $$f\left(\topint R_{x_0...x_N}\right)\cap \topint R_{y_0...y_N}\neq \emptyset$$ or equivalently (using Lemma \ref{lem-intersecting tilings gives tilings}) $$\bigcap_{i=0}^Nf^{-i}\left(\topint R_{x_i}\right) \cap \bigcap_{i=0}^Nf^{-i-1}\left(\topint R_{y_i}\right)\neq \emptyset.$$
    Then, by the disjointness of interiors of the partition elements of $\M$, we have that $x_i=y_{i-1}$ for $i=1,...,N$. A given $y\in \bigcap_{i=0}^Nf^{-i}\left(\topint R_{y_i}\right)$ then in particular satisfies $y\in\topint R_{y_0}$ so by the Markov property for $\M$, $y$ has a unique preimage $x\in\topint R_{x_0}$, meaning \begin{align*}
        x\in \topint R_{x_0}\cap f^{-1}\left(\topint R_{y_0}\right)\cap ... \cap f^{-N-1}\left(\topint R_{y_N}\right)&\subseteq \topint R_{x_0}\cap f^{-1}\left(\topint R_{x_1}\right) \cap...\cap f^{-N}\left(\topint R_{x_N}\right) \\
        &\subseteq R_{x_0...x_N}
    \end{align*}
    showing that $\M^N$ is a Markov partition for $f$.

    Finally, from Lemma \ref{lem-intersecting tilings gives tilings} and Lemma \ref{lem-boundary is invariant and nowhere dense} we have $$\partial\M^N=\bigcup_{i=0}^Nf^{-i}\partial\M=f^{-N}\partial\M$$ as desired.
\end{proof}

Since $f$ is locally an invertible linear map, $f^{-N}\partial\M=\partial\M^N$ implies that that if $\M$ is smooth (respectively linear), then so is $\M^N$. Moreover, $\partial\M$ and $\partial\M^N$ have the same Hausdorff dimension.

\subsection{Markov tilings}
The next step is to develop a notion corresponding to a Markov partition in the covering space of the torus. Here, the lift of an expanding toral endomorphism acts by a linear isomorphism, which will make constructive arguments easier. We will also use this concept to construct the linear Markov partitions for the forward direction of Theorem \ref{thm-A smooth MP iff A^k diagonalizable with integer eigenvals}.

More precisely, for an expanding integer matrix $A$ we will consider a tiling of $\R^d$ that also satisfies: \begin{enumerate}
    \item \textit{(Translation invariance)} For any $v\in \Z^d$, $\mathfrak{C}+v=\mathfrak{C}$.
    \item \textit{(Finiteness)} $\#\{C\in \mathfrak{C}:C\cap[0,1]^d\neq \emptyset\}<\infty.$
    \item \textit{(Smallness)} For all $C\in\mathfrak{C}$, $\diam(C)<1$ and $\operatorname{diam}(AC)<1.$
    \item \textit{(Markov property)} For all $C\in\mathfrak{C}$, there exist $n=n(C)\in \N\setminus\{0\}$, $C_1,...,C_n\in \mathfrak{C}$ satisfying $$AC=\bigcup_{1\leq i\leq n} C_i.$$ 
\end{enumerate}

The first two conditions imply that the latter two need only be checked for all of the finitely many $C\in \mathfrak{C}$ that intersect the unit cube. We will call a tiling with the properties above a \textit{Markov tiling} for $A$. The relation between Markov tilings for $A$ and Markov partitions for $f_A$ is given in the following proposition.

\begin{proposition}\label{prop:MPs give Markov tilings}
    Let $A$ be a $d\times d$ expanding integer matrix. Then the projection of any Markov tiling for $A$ is a Markov partition for $f_A$. 

    Conversely, any Markov partition for $f_A$ can, after a refinement, be lifted to a Markov tiling of $\R^d$.
\end{proposition}

\begin{proof}
    First, suppose $\mathfrak{C}$ is a Markov tiling for $A$. Then, by translation invariance and the smallness condition, $\operatorname{int}(\rho C)\cap \operatorname{int}(\rho C') \neq \emptyset$ if and only if $C=C'+v$ for some $v\in \Z^d$. Finiteness ensures that $\rho\mathfrak{C}$ consists of finitely many distinct elements, and the smallness property implies in particular that $f|_{\rho(C)}$ is injective for any $C\in \mathfrak{C}$. Finally, the Markov property for tilings means that $\rho\mathfrak{C}$ satisfies (MP+). It remains to show that $\rho\mathfrak{C}$ is metrically generating. Enumerate the elements of $\rho\mathfrak{C}$ so that $$\rho\mathfrak{C}=\{R_i\}_{i=1}^n,$$ and let $(x_i)_{i\in \N}\in \{1,...,n\}^{\N}$. Suppose $$x,y\in \bigcap_{i=0}^\infty f^{-i}(\operatorname{int}R_{x_i}),$$ and let $\bar{x},\bar{y}\in \R^d$ be two arbitrary lifts of $x,y$, and $C_{x_i}$ an arbitrary tile in the preimage of $R_{x_i}$. Then for all $i$, there exists some $v_i,w_i\in \Z^d$ such that $$\bar{x}+v_i,\bar{y}+w_i\in A^{-i}C_{x_i}.$$ Since $A$ is expanding, all eigenvalues of $A^{-1}$ have modulus less than 1. Moreover, the tiles of $\mathfrak{C}$ are uniformly bounded in diameter, so we get $$\operatorname{diam}(A^{-i}C_{x_i})\xrightarrow[i\to \infty]{} 0.$$ This means $\lim_{i\to\infty} d_\R(\bar{x}+v_i,\bar{y}+w_i)= 0,$ meaning that $$d_\T(\rho\bar{x},\rho\bar{y})=d_\T(x,y)=0,$$ so $x=y$. We conclude that $\rho\mathfrak{C}$ is a Markov partition satisfying (MP+).

    Now suppose $\mathcal{M}$ is a Markov partition for $f_A$. Note that for $x,y\in\T^d$ sufficiently close, $$ d_\T(f_Ax,f_Ay) \leq \|A\| d_\T(x,y),$$ where $\|A\|$ denotes the operator norm of $A$ induced by the Euclidian norm on $\R^d$. Thus, using Lemma \ref{lem-backwards refinements of MP} we can refine $\M$ such that it satisfies (MP+) and for all $R\in\mathcal{M}$, $$\diam(R)<\frac{1}{2\|A\|}<\frac{1}{2}.$$

    For each $R_i$, pick an arbitrary point $x_i\in R_i$ and $\bar{x}_i\in \R^d$ some lift. Then for each $v\in \Z^d$ and each $y\in R_i$, there exists a unique $\bar{y}\in \rho^{-1}(y)\cap B(\bar{x}_i+v,1/2).$ We define our tiling to be $$\mathfrak{C}:=\{\rho^{-1}R_i\cap B(\bar{x}_i+v,1/2):R_i\in\M,v\in\Z^d\}.$$
    That is, $\mathfrak{C}$ is the collection of unique copies of each $R_i$ within 1/2-balls centered at $\Z^d$ translations of $\bar{x}_i$.
    Translation invariance of $\mathfrak{C}$ is immediate, and the finiteness property follows from the fact that $\mathcal{M}$ is finite. Since $\diam(R_i)<(2\|A\|)^{-1}$, we also have $\diam(AC)<1/2<1$ for all $C\in \mathfrak{C}$. It remains to show that $AC=\bigcup_jC_j$ for some finite set $\{C_j\}\subset\mathfrak{C}$. 

    First note that $$\rho AC=f\rho C=\bigcup_{j=1}^kR_j$$ for some $R_j\in \M$. Since $\diam(AC)<1/2$, for each $R_j$ there is a unique $v_j\in\Z^d$ such that $\bar{x}_j+v_j\in AC$. Thus, $AC$ consists of exactly those unique copies of $R_j$ containing $\bar{x}_j+v_j$, hence $AC$ is a finite union of tiles $C_j\in \mathfrak{C}$. 
\end{proof}

One benefit of Markov tilings is that it is not difficult to turn a Markov tiling for one matrix into a Markov tiling for a similar matrix without changing its geometrical properties.

\begin{lemma}\label{lem:MPforsimilarmatrices}
    Let $A,B$ be $d\times d$ expanding integer matrices that are similar over $\Q$. Then for any Markov tiling $\mathfrak{C}$ for $A$ there exists a linear transformation turning $\mathfrak{C}$ into a Markov tiling for $B$.
\end{lemma}

\begin{proof}
    Since $A,B$ are similar over $\Q$, there exists an integer matrix $T$ invertible over $\Q$ such that $B=TAT^{-1}$, where $T^{-1}$ need not be an integer matrix. Let $K$ be a large enough multiple of $|\det T|$ such that $\frac{1}{K}T\mathfrak{C}$ satisfies the smallness property with respect to $B$, i.e. $$\operatorname{diam}\left(\frac{1}{K}TC\right)<1,\text{ and }\operatorname{diam}\left(B\left(\frac{1}{K} TC\right)\right)<1$$ for all of the finitely many $C\in \mathfrak{C}$ for which $\frac{1}{K}TC$ intersects the unit cube. 
    
    For any $C\in \mathfrak{C}$, we have $$B\left(\frac{1}{K} T\mathfrak{C}\right)=\frac{1}{K} TAT^{-1}TC=\frac{1}{K}T\bigcup_{i}C_i=\bigcup_i\frac{1}{K}TC_i,$$ so $\frac{1}{K}T\mathfrak{C}$ satisfies the Markov property for tilings.
    Moreover, for any $v\in \Z^d$, $$\frac{1}{K}T\mathfrak{C}+v=\frac{1}{K}T(\mathfrak{C}+KT^{-1}v).$$
    Since $|\det T|$ divides $K$, $KT^{-1}v\in \Z^d$ by Cramer's rule. $\mathfrak{C}$ is translation invariant, so $$\frac{1}{K}T\mathfrak{C}+v=\frac{1}{K}T\mathfrak{C}.$$

    We conclude that $\frac{1}{K}T\mathfrak{C}$ is a Markov tiling for $B$, as desired. 
\end{proof}

For the next lemma, we use something akin to the refinement described in Lemma \ref{lem-backwards refinements of MP}, but for Markov tilings.

\begin{lemma}\label{lem:MPforpowers}
    Let $\mathfrak{C}$ be a Markov tiling for $A^k$, then for some positive integer $K$, $$\mathfrak{B}:=\bigvee_{i=0}^{k-1}A^{-i}K^{-1}\mathfrak{C}$$ is a Markov tiling for $A$. 
    
    Conversely, if $\mathfrak{D}$ is a Markov tiling for $A$, then for some positive integer $M$, $M^{-1}\mathfrak{D}$ is a Markov tiling for $A^k$. 
    
\end{lemma}

\begin{proof}
    Let $K$ be a large enough integer so that $\mathfrak{B}$ satisfies the smallness property with respect to $A$. Since $A$ acts by homeomorphism on $\R^d$, $A^{-i}L^{-1}\mathfrak{C}$ is a tiling of $\R^d$ for each $i\in\{0,...,k-1\}$. By Lemma \ref{lem-intersecting tilings gives tilings}, $\mathfrak{B}$ is a tiling of $\R^d$. The finiteness property for $\mathfrak{B}$ is clear. 

    Fix $i\in\{0,...,k-1\}$, let $v\in \Z^d$ and $C\in \mathfrak{C}$. Then $$A^{-i}K^{-1}C+v=A^{-i}K^{-1}(C+A^{i}Kv),$$ but since $A^{i}Kv\in \Z^d$, we have that $C+A^{i}Kv=C'\in \mathfrak{C}$, so $$A^{-i}K^{-1}C+v=A^{-i}K^{-1}C'\in A^{-i}K^{-1}\mathfrak{C}.$$ 
    Next, since translation by $v\in\Z^d$ is a homeomorphism on $\R^d$, for $C_0,...,C_{k-1}\in \mathfrak{C}$ and $v\in\Z^d$, we have \begin{align*}
        &\overline{\operatorname{int}K^{-1}C_0\cap \operatorname{int}A^{-1}K^{-1}C_1\cap...\cap\operatorname{int}A^{-(k-1)}K^{-1}C_{k-1}} +v \\
        &= \overline{\operatorname{int}(K^{-1}C_0+v)\cap \operatorname{int}(A^{-1}K^{-1}C_1+v)\cap...\cap\operatorname{int}(A^{-(k-1)}K^{-1}C_{k-1}+v)} \\
        &= \overline{\operatorname{int}K^{-1}C_0'\cap \operatorname{int}A^{-1}K^{-1}C_1'\cap...\cap\operatorname{int}A^{-(k-1)}K^{-1}C_{k-1}'}
    \end{align*}
    where $C_i' \in \mathfrak{C}$ for $i\in\{0,1,...,k-1\}$, so $\mathfrak{B}$ is translation invariant.
    
    For the Markov property, we note that since $A$ is a homeomorphism on $\R^d$, we get for some finite index set $J$

    \begin{align*}
        A&\left(\overline{\operatorname{int}K^{-1}C_0\cap \operatorname{int}A^{-1}K^{-1}C_1\cap...\cap\operatorname{int}A^{-(k-1)}K^{-1}C_{k-1}}\right)\\
        &=\overline{\operatorname{int} AK^{-1}C_0\cap \operatorname{int} K^{-1}C_1\cap... \cap \operatorname{int}A^{-(k-2)}K^{-1}C_{k-1}} \\
        &=\overline{\topint\left(A^{-(k-1)}K^{-1}A^kC_0\right)\cap \operatorname{int} K^{-1}C_1\cap... \cap \operatorname{int}A^{-(k-2)}K^{-1}C_{k-1}}\\
        &= \overline{\topint\left(A^{k-1}K^{-1}\bigcup_{j\in J}C_j\right)\cap \operatorname{int} K^{-1}C_1\cap... \cap \operatorname{int}A^{-(k-2)}K^{-1}C_{k-1}}\\
        &= \bigcup_{j\in J} \overline{\operatorname{int} K^{-1}C_1\cap... \cap \operatorname{int}A^{-(k-2)}K^{-1}C_{k-1}\cap \topint A^{-(k-1)}K^{-1}C_j}
    \end{align*}
    We have shown that $\mathfrak{B}=\bigvee_{i=0}^{k-1}A^{-i}K^{-1}\mathfrak{C}$ is a Markov tiling for $A$.

    Suppose now that $\mathfrak{D}$ is a Markov tiling for $A$. Take $M$ large enough so that $M^{-1}\mathfrak{D}$ satisfies the smallness property with respect to $A^k$. Since the remaining properties are straightforward to show for $M^{-1}\mathfrak{D}$, we conclude that $M^{-1}\mathfrak{D}$ is a Markov tiling for $A^k$.
\end{proof}

Combining Lemma \ref{lem-backwards refinements of MP}, Proposition \ref{prop:MPs give Markov tilings}, Lemma \ref{lem:MPforsimilarmatrices}, Lemma \ref{lem:MPforpowers}, and the observation that $\rho:\R^d\to \T^d$ is locally a linear bijection proves Proposition \ref{prop-reduction to similarity and powers}. In fact, the procedures described above also preserve Hausdorff dimension.

Using Markov tilings, we can also construct very simple linear Markov partitions for diagonal expanding integer matrices.

\begin{lemma}\label{lem:MP for diagonal}
    Let $D$ be a $d\times d$ expanding diagonal integer matrix. Then $D$ admits a linear Markov partition.
\end{lemma}

\begin{proof}
    Let $K\geq |\det D|$ such that $$\operatorname{diam}(D[0,1/K]^d)<1.$$ Then $$\mathfrak{C}=\{[i_1/K,(i_1+1)/K]\times ... \times [i_d/K,(i_d+1)/K]:i_1,..,i_d\in\Z\}$$
    is easily shown to be a Markov tiling for $D$. 
\end{proof}

Using this construction, the backward direction of Theorem \ref{thm-A smooth MP iff A^k diagonalizable with integer eigenvals} becomes immediate. 

\begin{proof}[Proof of Theorem \ref{thm-A smooth MP iff A^k diagonalizable with integer eigenvals} ($\impliedby$)]
    Combine Lemma \ref{lem:MP for diagonal} and Proposition \ref{prop-reduction to similarity and powers}.
\end{proof}

\section{Non-smoothness}\label{Section: Non-smoothness}

To show the converse direction of Theorem \ref{thm-A smooth MP iff A^k diagonalizable with integer eigenvals}, we will use Mañé's result, Lemma \ref{lem-Mane}, and a somewhat similar argument as two base cases. The second case in particular handles non-diagonalizable matrices with integer eigenvalues, where the conditions of Lemma \ref{lem-Mane} are clearly not met. We remind the reader that we abbreviate $\gamma([0,1])$ to just $\gamma$ for legibility.

\begin{proposition}\label{prop-Jordan block}
    Let $B=\begin{pmatrix}
        k & 1 \\ 0 & k
    \end{pmatrix}$ for some integer $k>1$. Let $\gamma:[0,1]\to \R^d$ be a $C^1$ curve where $\gamma'_2(t)$ is non-zero for some $t\in[0,1]$, and $p\in \N\setminus \{0\}$. Then for any $\varepsilon>0$ there exists an $n\in \N$ such that $\rho B^{pn}\gamma$ is $\varepsilon$-dense in $\T^2$. 
\end{proposition}

\begin{proof} 
    Let $\{(x_1,y_1),...,(x_K,y_K)\}\subset \R^2$ be a lift of an $\varepsilon$-dense set of points in $\T^2$. Since $\gamma$ is $C^1$ and $\gamma_2'$ is non-vanishing somewhere, we may pass to a subcurve of $\gamma$ and possibly reverse the orientation to assume that $\gamma_2'(t)>0$ for all $t\in[0,1]$. By compactness and continuity, we can find $M,m>0$ such that $$|\gamma'_1(t)|\leq M \text{ and } 0<m<\gamma_2'(t)\leq M$$ for all $t$. Also, $\gamma_2([0,1])$ contains an open interval $I$.

    Since the matrix $$\begin{pmatrix}
        k^p &  & \\
        &\ddots & \\
        & & k^p
    \end{pmatrix}$$ acts transitively on $\T^K$, there exists a point $$(t_1,...,t_K)\in \left(\gamma_2^{-1}I\right)^K$$ and a sequence of integers $n_j\to\infty$ so that for all $i\in\{1,...,K\}$ and $j\in \N$, $$d_\T(k^{pn_j}\gamma_2(t_i),y_i)<\varepsilon.$$

    Fix $$n=n_j> \frac{1+\varepsilon}{p\varepsilon\frac{m}{kM}}$$ such that with $$\eta_n:=\frac{\varepsilon}{k^{pn}M}$$ we have $$t_i+\eta_n\in [0,1]$$ for each $i$. 
    
    This gives us \begin{align*}
        B^{pn}(\gamma(t_i+\eta_n)-\gamma(t_i))_1&=\int_{t_i}^{t_i+\eta_n} B^{pn}\gamma'(s)_1ds \\
        &=\int_{t_i}^{t_i+\eta_n} k^{pn}\gamma_1'(s)+pnk^{pn-1}\gamma_2'(s)ds \\
        &\geq -\eta_nk^{pn}M+\eta_npnk^{pn-1}m \\
        &= -\varepsilon + p\varepsilon n \frac{m}{kM} > 1
    \end{align*}
    By continuity and the Intermediate Value Theorem, there exist $\tilde{t}_i\in [t_i,t_i+\eta_n]$ such that $$B^{pn}(\gamma(\tilde{t}_i)-\gamma(t_i))_1=x_i-B^{pn}\gamma(t_i)_1 \mod 1.$$ Moreover, \begin{align*}
        \left|B^{pn}(\gamma(\tilde{t}_i)-\gamma(t_i))_2\right| &\leq \int_{t_i}^{\tilde{t}_i}\left|B^{pn}\gamma'(s)_2\right|ds \\
        &= \int_{t_i}^{\tilde{t}_i} |k^{pn}\gamma_2'(s)|ds \\
        &\leq (\tilde{t}_i-t_i)k^{pn}M \\
        &\leq \eta_nk^{pn}M = \varepsilon
    \end{align*}

    For any $i\in \{1,...,K\}$, we now have \begin{align*}
        d_\T\left((x_i,y_i),B^{pn}\gamma(\tilde{t}_i)\right) &\leq d_\T(x_i,B^{pn}\gamma(\tilde{t}_i)_1)+d_\T\left(y_i,B^{pn}\gamma(\tilde{t}_i)_2\right) \\
        &\leq 0 + d_\T\left(y_i,B^{pn}\gamma(t_i)_2\right) + d_\T\left(B^{pn}\gamma(t_i)_2,B^{pn}\gamma(\tilde{t}_i)_2\right) \\
        &\leq d_\T\left(y_i,k^{pn}\gamma_2(\tilde{t}_i)\right) + \left|B^{pn}\left(\gamma(\tilde{t}_i)-\gamma(t_i)\right)\right| \\
        &\leq 2\varepsilon
    \end{align*}

    In other words, the set $$\rho\left\{B^{pn}\gamma(\tilde{t}_i):i\in\{1,...,K\}\right\}\subseteq \rho B^{pn}\gamma$$ is $3\varepsilon$-dense in $\T^2$, and we are done.

\end{proof}

To reduce to Lemma \ref{lem-Mane} or the proposition above, we will need the following lemma. Essentially, it states that for any smooth Markov partition, its boundary cannot be completely parallel to an invariant space.

\begin{lemma} \label{lem-open projection lemma}
    Let $$A=\begin{pmatrix}
        B & D \\ 0 & C
    \end{pmatrix}$$ be a $d\times d$ expanding integer matrix with $B$ of dimension $d'\times d'$, $\mathcal{M}$ a smooth Markov partition for $A$, and denote by $E_B$ the invariant subspace of $B$. Then there exists a $C^1$ hypersurface $\mathcal{V}\subset \R^d$ such that $\rho(\mathcal{V})\subset \partial\M$ and $\operatorname{proj}_{E_B^\perp}(\mathcal{V})$, that is the projection of $\mathcal{V}$ to $E_B^\perp$ along $E_B$, contains an open set $\mathcal{U}$ where for every $y\in \mathcal{U}$, $y+E_B$ intersects $\mathcal{V}$ transversally. 
\end{lemma}

\begin{proof}
    Since $\M$ is smooth, we can write $\partial\M=\bigcup_{i=1}^l\tilde{\mathcal{S}}_i$ where each $\tilde{\mathcal{S}}_i$ is a $C^1$ hypersurface. For each $i$, let $\mathcal{S}_i\subset \R^d$ be a lift of $\tilde{\mathcal{S}}_i$. To reach the conclusion, it suffices to find an $i$ and $x\in \mathcal{S}_i$ such that $\vec{N}_x\mathcal{S}_i$ is not parallel to $E_B^\perp$, where $\vec{N}$ is some local choice of unit normal vector.

    Suppose towards a contradiction that this is not true. Then for every $i$ and for every $x\in \mathcal{S}_i$, $\vec{N}_x\mathcal{S}_i$ is 0 in the first $d'$ coordinates. Recall that if $P$ is a hyperplane in $\R^d$ with normal vector $\vec{N}$ and $A$ is an invertible $d\times d$ matrix, then $$w\in A^{-1}P\iff Aw\in P \iff (Aw)\cdot \vec{N} = 0 \iff w \cdot (A^T\vec{N})=0.$$
    
    Thus, for any $y\in A^{-1}\mathcal{S}_i$ we have $$\vec{N}_y(A^{-1}\mathcal{S}_i)=\alpha A ^T \vec{N}_{Ay}\mathcal{S}_i$$ for some normalizing scalar $\alpha$. Because $$A^T=\begin{pmatrix}
        B^T & 0 \\ D^T & C^T
    \end{pmatrix},$$ this means $\vec{N}_y(A^{-1}\mathcal{S}_i)$ is also 0 in the first $d'$ coordinates, and hence parallel to $E_B^\perp$. 

    If $\mathcal{H}$ is an embedded hypersurface whose normal vector is parallel to $E_B^\perp$, the Rank Theorem (Theorem 7.1 in  \cite{Boo}) implies that $\operatorname{proj}_{E_B^\perp}(\mathcal{H})$ is a countable union of nowhere dense sets in $E_B^\perp$. As such, the set $$\mathcal{Z}:=\operatorname{proj}_{E_B^\perp}\left(\bigcup_{v\in\Z^d}\bigcup_{i=1}^l \mathcal{S}_i\cup A^{-1}\mathcal{S}_i+v\right)$$ is a countable union of nowhere dense sets, and so by the Baire Category Theorem, does not cover all of $E_B^\perp$. Let $\begin{pmatrix}
        \vec{0} \\ \vec{y}
    \end{pmatrix}\in E_B^{\perp}\setminus \mathcal{Z}$, and consider $$\mathcal{W}:=\left\{\begin{pmatrix}
        \vec{x} \\ \vec{0}
    \end{pmatrix}:\vec{x}\in[0,1)^{d^\prime}\right\}.$$ By our choice of $\begin{pmatrix}
        \vec{0} \\ \vec{y}
    \end{pmatrix}$, we have $$\rho\left(\mathcal{W}+\begin{pmatrix}
        \vec{0} \\ \vec{y}
    \end{pmatrix}\right)\cap \partial\M=\emptyset,$$ and since the elements of $\M$ have disjoint interiors, $$\rho\left(\mathcal{W}+\begin{pmatrix}
        \vec{0} \\ \vec{y}
    \end{pmatrix}\right)\subset \topint R$$ for a unique $R\in \M$. By the same argument, $$\rho \left(A\left(\mathcal{W}+\begin{pmatrix}
        \vec{0}\\\vec{y}
    \end{pmatrix}\right)\right)=\rho\left(\left\{\begin{pmatrix}
        B\vec{x}+D\vec{y} \\ C\vec{y}
    \end{pmatrix}:\vec{x}\in [0,1)^{d^\prime}\right\}\right)\subset \topint R'$$ for a unique $R'\in \mathcal{M}$. 
    
    Since $B$ is expanding, there exist $\vec{v}\neq \vec{w}\in [0,1)^{d^\prime}$ such that $\rho B\vec{v}=\rho B\vec{w}$. This means $$\rho\begin{pmatrix}
        \vec{v}\\\vec{y}
    \end{pmatrix}\neq \rho\begin{pmatrix}\vec{w} \\ \vec{y}\end{pmatrix}\in \topint R,$$ but $$f_A\rho\begin{pmatrix}
        \vec{v}\\\vec{y}
    \end{pmatrix} = \rho\begin{pmatrix}
        B\vec{v} + D\vec{y} \\ C\vec{y}
    \end{pmatrix}=\rho\begin{pmatrix}
        B\vec{w} + D\vec{y} \\ C\vec{y}
    \end{pmatrix}=f_A\rho\begin{pmatrix}
        \vec{w} \\ \vec{y}
    \end{pmatrix}\in\topint R'.$$

    This cannot be, since it contradicts the Markov property (MP), and the proof is complete.
\end{proof}

We now have all the necessary tools to finish proving our main result. 

\begin{proof}[Proof of Theorem \ref{thm-A smooth MP iff A^k diagonalizable with integer eigenvals} ($\implies$)]
Suppose towards a contradiction that $A$ is a $d\times d$ expanding integer matrix that has at least one eigenvalue that is not a root of an integer and that $\M$ is a smooth Markov partition for $A$. We will show that if $A$ has an eigenvalue $\lambda$ that is not a root of an integer, there exists an $m$ such that $$A^m\sim \begin{pmatrix}
    B & D \\
    0 & C
\end{pmatrix}$$ where $B$ (and hence any power of $B$) is of dimension greater than 1 and satisfies the assumptions of Lemma \ref{lem-Mane}. 

Write $\operatorname{Min}(\lambda)$ for the minimal polynomial over $\Q$ of $\lambda$, and note that for all $k\in\N$, $\deg\operatorname{Min}(\lambda)\geq \deg\operatorname{Min}(\lambda^k)\geq 1$, meaning that there exists an $m$ such that $\deg\operatorname{Min}(\lambda^{mn})$ is constant in $n$. By e.g. Theorem III.12 of \cite{N}, $A^m$ is then similar over $\Q$ (in fact over $\Z$) to $$\begin{pmatrix}
    B & D \\ 0 & C
\end{pmatrix}$$ where the characteristic polynomial of $B$ is $\operatorname{Min}(\lambda^m)$. This means $\chi_{B^n}$ is irreducible for all $n$, so $B$ satisfies the assumptions of Lemma \ref{lem-Mane}. Moreover, since $\lambda$ is not a root of an integer, $B$ has dimension $d'>1$.

Thus, by Proposition \ref{prop-reduction to similarity and powers}, we may assume $A$ itself is of this form. If $C$ is empty, that is $A=B$, let $\gamma(t)$ be any $C^1$ curve in $\partial\M$. Then, $\gamma$ is certainly rectifiable, so by Lemma \ref{lem-Mane}, $$\bigcup_{n\in\N} f_A^n\gamma\subset \bigcup_{n\in \N}f_A^n(\partial\M)=\partial\M$$ is dense in $\T^d$, contradicting Lemma \ref{lem-boundary is invariant and nowhere dense}.

Otherwise, let $\mathcal{V}\subset \R^d$ and $\mathcal{U} \subset E_B^\perp$ be the open sets from the conclusion of Lemma \ref{lem-open projection lemma}. We will show that $\partial\M$ is dense in the open set $$\rho(\mathcal{U}+E_B).$$

So, take $$\begin{pmatrix}
    \vec{x} \\ \vec{0} \end{pmatrix}\in E_B \text{ and }\begin{pmatrix} \vec{0} \\ \vec{y}
\end{pmatrix}\in \mathcal{U}$$ arbitrary and fix $\varepsilon>0$. Since periodic points for $f_C:\T^{d-d'}\to \T^{d-d'}$ are dense, there exists a $\begin{pmatrix}0 \\\vec{z}\end{pmatrix}\in \mathcal{U}$ such that $d_\T(\vec{y},\vec{z})<\varepsilon$ and $\rho\vec{z}$ is periodic under $f_C$ with period $p$. 

By transversality, $$\left((\vec 0,\vec z)+E_B\right) \cap \mathcal{V}$$ is a $C^1$ manifold of dimension $d-1$, and thus contains a $C^1$ curve $\begin{pmatrix} \vec\gamma(t) \\ \vec {z} \end{pmatrix}$. This $\vec{\gamma}(t)$ is itself a $C^1$ curve and hence rectifiable, so by Lemma \ref{lem-Mane}, there exists an $n\in \N$ such that $\rho B^{pn}\vec\gamma$ is $\varepsilon$-dense in $\T^{d'}$. In particular, there is some $t\in[0,1]$ for which $$d_\T(B^{pn}\vec{\gamma}(t),\vec{x}-\vec{q})<\varepsilon$$ where $$\vec{q}=\sum_{i=0}^{pn}B^iDC^{pn-i}\vec{z}.$$ 

Combining this with the fact that $$A^{pn}=\begin{pmatrix}
    B^{pn} & \sum_{i=0}^{pn} B^iDC^{pn-i} \\0 & C^{pn}
\end{pmatrix}$$

we obtain \begin{align*}
    d_\T\left(A^{pn}\begin{pmatrix}
        \vec{\gamma}(t) \\\vec{z}
    \end{pmatrix},\begin{pmatrix}
        \vec{x} \\\vec{y}
    \end{pmatrix}\right) &= d_\T\left(\begin{pmatrix}
        B^{pn}\vec{\gamma}(t) + \vec{q} \\C^{pn}\vec{z}
    \end{pmatrix},\begin{pmatrix}
        \vec{x} \\\vec{y}
    \end{pmatrix}\right) \\
    &\leq d_\T(B^{pn}\gamma(t) + \vec{q},\vec{x}) + d_\T(C^{pn}\vec{z},\vec{y}) \\
    &= d_\T(B^{pn}\gamma(t),\vec{x}-\vec{q})+d_\T(\vec{z},\vec{y}) \\
    &\leq 2\varepsilon
\end{align*}

Where $$\rho A^{pn}\begin{pmatrix}\vec\gamma(t) \\ \vec{z} \end{pmatrix} \in f_A^{pn}\rho\mathcal{V} \subset f_A^{pn}\partial\M\subset \partial\M.$$ We conclude that $\partial\M$ is dense in $\rho(\mathcal{U}+E_B)$, which contradicts Lemma \ref{lem-boundary is invariant and nowhere dense}.

It remains to handle the case where all eigenvalues of $A$ are roots of integers, but $A$ is not diagonalizable. The argument will be similar but slightly more intricate than before. Using the Jordan canonical form and once more Proposition \ref{prop-reduction to similarity and powers}, we may assume $$A=\begin{pmatrix}
    k & 1 & 0 \\
    0 & k & D\\
    0 & 0 & C
\end{pmatrix},$$ where $k$ is a positive integer. We allow $C$ (and hence $D$) to be empty, but if it is not, then $D$ is either all 0 or $\begin{pmatrix} 1 & 0 & ... & 0\end{pmatrix}$. 

Suppose $\mathcal{M}$ is a smooth Markov partition for $A$. Denote by $E_1$ the first coordinate axis, and let $\mathcal{V}\subset\R^d$ and $\mathcal{U}\subset E_1^\perp$ be the open sets from the conclusion of Lemma \ref{lem-open projection lemma}. We will again show that $\partial\M$ is dense in $\rho(\mathcal{U}+E_1)$, so take $$\begin{pmatrix}
    x \\
    0 \\
    \vec{0}
\end{pmatrix}\in E_1 \text{ and } \begin{pmatrix}
    0 \\ y_1 \\ \vec{y}_2
\end{pmatrix}\in \mathcal{U}.$$

Now, there is a point $$\begin{pmatrix}
    0 \\ z_1 \\ \vec{z}_2
\end{pmatrix}\in \mathcal{U}$$ such that $\rho\begin{pmatrix}
    z_1 \\ \vec{z_2}
\end{pmatrix}$ is periodic with period $p$ under the expanding $(d-1)$-dimensional toral endomorphism induced by $\begin{pmatrix}
    k & D \\ 0 & C
\end{pmatrix}$ and $d_\T\left(\begin{pmatrix}
    y_1\\ \vec{y}_2
\end{pmatrix}, \begin{pmatrix}
    z_1 \\ \vec{z}_2
\end{pmatrix}\right)<\varepsilon$. 

Moreover, since $$\begin{pmatrix}
    0 \\ z_1 \\ \vec{z}_2
\end{pmatrix}+E_1$$ intersects $\mathcal{V}$ transversally, the same is true for 
$$\begin{pmatrix}
    0 \\ z_1 \\ \vec{z}_2
\end{pmatrix}+E_1\oplus E_2,$$ where $E_2$ denotes the second coordinate axis. We obtain a $C^1$ curve $$\begin{pmatrix}
    \gamma_1(t) \\ \gamma_2(t) + z_1\\ \vec{z}_2 
\end{pmatrix}$$ contained in  $$\left(\begin{pmatrix}
    0 \\ z_1 \\ \vec{z}_2
\end{pmatrix}+E_1\oplus E_2\right) \cap \mathcal{V}.$$ Also, $\gamma_2'(t)$ is not constantly 0, by transversality of $\mathcal{V}$ with $$\begin{pmatrix}
    0 \\ z_1 \\ \vec{z}_2
\end{pmatrix}+E_1.$$

In particular, $\begin{pmatrix}
    \gamma_1(t) \\ \gamma_2(t)
\end{pmatrix}$ satisfies the assumptions of Proposition \ref{prop-Jordan block}, giving an $n\in \N$ for which $\rho \begin{pmatrix}
        k & 1 \\ 0 & k
    \end{pmatrix}^{pn}\begin{pmatrix}
    \gamma_1(t) \\ 
    \gamma_2(t)
\end{pmatrix}$ is $\varepsilon$-dense in $\T^2$. Thus, there exists a $t\in [0,1]$ such that $$d_\T\left(\begin{pmatrix}
        k & 1 \\ 0 & k
    \end{pmatrix}^{pn}\begin{pmatrix}
    \gamma_1(t) \\ 
    \gamma_2(t)
\end{pmatrix},\begin{pmatrix}
    x - q \\ y_1 - z_1
\end{pmatrix}\right)<\varepsilon$$ where $$q = pnk^{pn-1}z_1 + \sum_{i=0}^{pn-1}ik^iDC^{pn-1-i}\vec{z}_2.$$

Note that $$A^{pn}=\begin{pmatrix}
    k^{pn}  & pnk^{pn-1} & \sum_{i=1}^{pn-1}ik^iDC^{pn-1-i} \\
    0 & k^{pn} & \sum_{i=0}^{pn-1}k^iDC^{pn-1-i} \\
    0 & 0 & C^{pn}
\end{pmatrix}.$$

We now have \begin{align*}
    d_\T\left(A^{pn}\begin{pmatrix}
        \gamma_1(t) \\ \gamma_2(t) + z_1 \\\vec{z}_2
    \end{pmatrix},\begin{pmatrix}
        x \\ y_1 \\ \vec{y}_2
    \end{pmatrix}\right) &= d_\T\left(\begin{pmatrix}
        k^{pn}\gamma_1(t) + pn k^{pn-1}\gamma_2(t)\\ k^{pn} \gamma_2(t) \\
        \vec{0}
    \end{pmatrix} + \begin{pmatrix}
        q \\ z_1 \\ \vec{z}_2
    \end{pmatrix},\begin{pmatrix}
        x \\ y_1 \\ \vec{y}_2
    \end{pmatrix}\right) \\
    &\leq d_\T\left(\begin{pmatrix}
        k & 1 \\ 0 & k
    \end{pmatrix}^{pn}\begin{pmatrix} \gamma_1(t) \\ \gamma_2(t) \end{pmatrix}, \begin{pmatrix}
        x-q \\ y_1-z_1
    \end{pmatrix}\right) + d_\T\left(\vec{z}_2,\vec{y}_2\right) \\
    &\leq \varepsilon + \varepsilon = 2\varepsilon
\end{align*}

As before, since $$\rho A^{pn}\begin{pmatrix}
    \gamma_1(t) \\ \gamma_2(t) + z_1 \\
    \vec{z}_2
\end{pmatrix}\in f_A^{pn}\rho(\mathcal{V})\subset f_A^{pn}\partial\M\subset\partial\M,$$ we conclude that $\partial\M$ is dense in the open set $\rho(\mathcal{U}+E_1)$, contradicting Lemma \ref{lem-boundary is invariant and nowhere dense} and completing the proof of Theorem \ref{thm-A smooth MP iff A^k diagonalizable with integer eigenvals}.

\end{proof}

\section{On linearity}\label{Section: linearity}

\subsection{Smooth Markov partitions are linear}

To prove Theorem \ref{thm-smooth is linear}, we will use an argument very similar to the one in Section \ref{Section: Non-smoothness}. Lemma \ref{lem-nonlinear curves for diagonal matrices are dense} will provide a base case stating that non-linear curves have dense orbits in dimension 2, and Lemma \ref{lem-nonlinear surfaces have nonlinear curves} will allow a reduction to the base case. 

Note that Lemma \ref{lem-nonlinear curves for diagonal matrices are dense} also restricts what linear Markov partitions can look like for diagonal integer matrices in dimension 2. If the eigenvalues are the same modulus, any line segment in the boundary must have rational slope. If the eigenvalues have different modulus, then every line segment in the boundary must be aligned with one of the coordinate axes. 

\begin{lemma}\label{lem-nonlinear curves for diagonal matrices are dense}
    Let $B=\begin{pmatrix}
        k & 0 \\ 0 & l
    \end{pmatrix}$ with integers $k\geq l>1$, $\gamma$ a $C^1$ curve in $\R^2$, and $\varepsilon>0$. Then \begin{enumerate}
        \item If $k=l$ and $\gamma'(t)$ has irrational slope for some $t\in [0,1]$, then there exists an $n\in\N$ such that $\rho B^{n}\gamma$ is $\varepsilon$-dense in $\T^2.$
        \item If $k>l$ and $\gamma'(t)\notin \R e_1\cup \R e_2$ for some $t\in [0,1]$, then there exists an $n\in\N$ such that $\rho B^{n}\gamma$ is $\varepsilon$-dense in $\T^2.$ 
    \end{enumerate}

    In particular, if $\gamma$ is a non-linear $C^1$ curve in $\R^2$, then $\rho B^n\gamma$ is $\varepsilon$-dense in $\T^2$.
\end{lemma}

\begin{proof}
    We prove the two cases separately. 
    \begin{enumerate}
        \item Without loss of generality, we may pass to a subcurve and assume $\gamma'(0)$ has irrational slope. Let $$\vec{v}:=\frac{\gamma'(0)}{\|\gamma'(0)\|},$$ where $\|\cdot\|$ refers to the Euclidian norm on $\R^2$. Let also $\varepsilon>0$, and take $N>0$ so that $$\rho([0,N]\vec{v})=\rho\left(\left\{t\vec{v}:t\in[0,N]\right\}\right)$$ is $\varepsilon$-dense in $\T^2$. We will write $\operatorname{proj}_{\vec{v}}$ to mean the orthogonal projection onto $\R\vec{v}$ and $\theta(\vec{u},\vec{w})$ for the (unsigned) angle between two non-zero vectors $\vec{u},\vec{w}$. Observe that $B$ is a scalar matrix, so for any $n\in \N$, $$\theta(B^n\vec{u},\vec{w})=\theta(\vec{u},\vec{w}).$$

        By definition of $\gamma'$, there exists an $\eta>0$ such that for all $t\in(0,\eta]$, $$\theta(\gamma(t)-\gamma(0),\vec{v})<\frac{\pi}{2}$$ and $$\tan(\theta(\gamma(t)-\gamma(0),\vec{v}))<\varepsilon/N.$$ In particular, we have $\gamma(\eta)\neq\gamma(0)$. 

        Since $B$ is expanding, there exists an $n\in\N$ such that $$\|\operatorname{proj}_{\vec{v}}(B^{n}(\gamma(\eta)-\gamma(0)))\|>N.$$ Fix some arbitrary $\vec{x}\in[0,1)^2$ and take $r\in[0,N]$ such that $$d_\T(\vec{x}-B^{n}\gamma(0),r\vec{v})<\varepsilon.$$ Then by continuity and the Intermediate Value Theorem, there is a $t\in[0,\eta]$ such that $$\operatorname{proj}_{\vec{v}}(B^{n}(\gamma(t)-\gamma(0)))=r\vec{v}.$$ 

        Now, for $r \neq 0$, using the fact that $B$ preserves angles and the projection is orthogonal, we get \begin{align*}
            d_\T(\vec{x},B^{n}\gamma(t)) &\leq d_\T\left(\vec{x},r\vec{v}+B^n\gamma(0)\right)+d_\T(r\vec{v}+B^n\gamma(0),B^n\gamma(t)) \\ 
            &\leq \varepsilon + d_\R\big(B^n(\gamma(t)-\gamma(0)),\operatorname{proj}_{\vec{v}}B^n(\gamma(t)-\gamma(0))\big) \\
            &\leq \varepsilon + \|\operatorname{proj}_{\vec{v}}B^n(\gamma(t)-\gamma(0))\|\tan\left(\theta(B^n(\gamma(t)-\gamma(0)),\vec{v})\right) \\
            &\leq \varepsilon + r \tan\theta(\gamma(t)-\gamma(0), \vec{v}) \\
            &\leq \varepsilon + N\frac{\varepsilon}{N} = 2\varepsilon
        \end{align*}
        When $r = 0$, we of course take $t=0$ and reach the same desired inequality. This shows that $\rho B^n\gamma$ is $2\varepsilon$-dense in $\T^2$.

        \item We will use an argument similar to the proof of Proposition \ref{prop-Jordan block}. Since $\gamma'(t)\notin \R e_1 \cup \R e_2$ for some $t$, we may pass to a subcurve (possibly reversing the orientation) and assume that $\gamma'_1(t)>0$ for all $t\in[0,1]$. Moreover, we have $M,m>0$ such that $$\gamma_1'(t)>m \text{ and } |\gamma'_2(t)|\leq M$$ for all $t$. Take an arbitrary set of points $\{(x_1,y_1),...,(x_K,y_K)\}\subset \R^2$ whose image is $\varepsilon$-dense in $\T^2$. As in Proposition \ref{prop-Jordan block}, we can find distinct $t_1,...,t_K\in [0,1]$ and a sequence of integers $n_j\to\infty$ such that for all $i$ and $j$, $$d_\T\left(y_i,l^{n_j}\gamma_2(t_i)\right)<\varepsilon.$$

        Fix $n=n_j$ so that $$\frac{l^n}{k^n}<\frac{m\varepsilon}{M}$$  and with $\eta_n:= \frac{1}{k^nm}$, $$t_i+\eta_n\in [0,1]$$ for each $i$. We get $$B^{n}(\gamma(t_i+\eta_n)-\gamma(t_i))_1=\int_{t_i}^{t_i+\eta_n}k^n\gamma_1'(s)ds \geq \eta_nk^nm=1,$$ so by the Intermediate Value Theorem, for each $i$ we have a $\tilde{t}_i\in [t_i,t_i+\eta_n]$ for which $$B^n(\gamma(\tilde{t}_i)-\gamma(t_i))_1=x_i-B^n\gamma(t_i)_1 \mod 1.$$

        Then \begin{align*}
            d_\T\left(B^n\gamma(\tilde{t}_i),(x_i,y_i)\right) &\leq d_\T\left(B^n\gamma(\tilde{t}_i)_1,x_i\right) + d_\T\left(B^n\gamma(\tilde{t}_i)_2,y_i\right) \\
            &\leq 0 + d_\T\left(B^n\gamma(\tilde{t}_i)_2,B^n\gamma(t_i)_2\right) + d_\T\left(B^n\gamma(t_i)_2,y_i\right) \\
            &\leq \left|\int_{t_l}^{\tilde{t}_l}l^n\gamma_2'(s)ds\right| +  d_\T(l^n\gamma_2(t_i),y_i) \\
            &\leq (\tilde{t}_i-t_i)l^nM+\varepsilon \\
            &\leq \eta_nl^nM+\varepsilon=2\varepsilon
        \end{align*}
        In conclusion, $$\rho\left\{B^n\gamma(\tilde{t}_i):i\in\{1,...,K\}\right\}\subseteq \rho B^n\gamma$$ is $3\varepsilon$-dense in $\T^2$, and we are done.
        
    \end{enumerate}

\end{proof}

In what follows, we will write $\{e_i\}_{i=1}^d$ for the standard basis of $\R^d$ and $E_{i,j}:=\R e_i\oplus\R e_j$.

\begin{lemma}\label{lem-nonlinear surfaces have nonlinear curves}
    Let $\mathcal{S}$ be a non-linear $C^1$ hypersurface in $\R^d$ with $d\geq 3$. Then there exists a hypersurface $\mathcal{V}\subset \mathcal{S}$, indices $i\neq j \in \{1,...,d\}$, and an open subset $\mathcal{U}\subset E_{i,j}^\perp$ such that for all $w\in \mathcal{U}$, $\left(w+E_{i,j}\right)\cap \mathcal{V}$ contains a non-linear $C^1$ curve $\gamma$.
\end{lemma}

\begin{proof}
    Let $u\mapsto\vec{N}_u\mathcal{S}$ be a (local) choice of continuous unit length normal vector to $\mathcal{S}$. By non-linearity of the hypersurface, there exists some $x_0 \in \mathcal{S}$ such that on every small neighbourhood $\mV \subset \mathcal{S}$ of $x_0$, there are points $x,y \in \mV$ such that $\vec N_x \mS \neq \vec N_y \mS$. By restricting $\mathcal{S}$ to one such neighbourhood that is sufficiently small and by reordering the standard basis vectors, we can parametrize  $\mV$ as the graph of a $C^1$ function 
    $$
        (x_1,\ldots,x_{d-1}) \mapsto f(x_1,\ldots,x_{d-1}) \in \R.
    $$
    For convenience, we will take our domain to be of the form $\prod_{i=1}^{d-1}I_i$ where each $I_i$ is a bounded open interval. 

    Suppose that for some $j\in\{1,...,d-1\}$ and some $$(a_1,...,a_{j-1},a_{j+1},...,a_{d-1})\in \prod_{i\neq j} I_i,$$ the partial derivative $$\frac{\partial f}{\partial x_j} (a_1,...a_{j-1},x_j,a_{j+1},...,a_{d-1})$$ is non-constant in $x_j$. Then by continuity of the derivatives, the same is true for a small neighborhood $\mathcal{U}$ of $(a_1,...a_{j-1},a_{j+1},...,a_{d-1})$ in $\prod_{i\neq j} I_i$, and we have that for all $$(b_1,...,b_{j-1},b_{j+1},...,b_{d-1})\in \mathcal{U},$$ the set $$\big(E_{j,d} + (b_1,...,b_{j-1},0,b_{j+1},...,b_{d-1},0)\big)\cap \mathcal{V}$$ contains a non-linear curve described by 
    $$
        x_j \mapsto f(b_1,...,b_{j-1},x_j,b_{j+1},...,b_{d-1})
    $$ as $x_j\in I_j$ varies. This gives us the desired result.

    Suppose now that $\partial f/\partial x_j$ is constant in $x_j$ for all $j$. We claim that $f$ is linear in all its variables, so $$f=\sum_{J\subset \{1,...,d-1\}} m_Jx_J$$ where $m_J\in \R$, $x_J=\prod_{j\in J} x_j$ and $x_\emptyset=1$. Indeed, if $g$ is some $C^1$ function in one variable $y$ with constant derivative, then $g(y)=my+c$, which we will use to prove the above claim by induction. Since $\partial f/\partial x_{d-1}$ is constant in $x_{d-1}$, $$f(x_1,...,x_{d-1})=m(x_1,...,x_{d-2})\cdot x_{d-1}+c(x_1,...,x_{d-2}).$$ Further, for any $a\neq a'\in I_{d-1}$, $$f(x_1,...,x_{d-2},a)-f(x_1,...,x_{d-2},a')=(a-a')m(x_1,...,x_{d-2}),$$ so $m,$ and hence $c$, are $C^1$. Moreover, for $i\in \{1,...,d-2\}$, $$\frac{\partial m}{\partial x_i}(x_1,...,x_{d-2})=\frac{1}{a-a'}\left(\frac{\partial f}{\partial x_i} (x_1,...,x_{d-2},a)-\frac{\partial f}{\partial x_i} (x_1,...,x_{d-2},a')\right),$$ so $\partial m / \partial x_i$ is constant in $x_i$, and the same holds for $c$. Applying the inductive step to $m$ and $c$ now yields $$f=\sum_{J\subset \{1,...,d-1\}} m_Jx_J$$ as desired.

    Moreover, because $\mathcal{V}$ is non-linear, at least one of $\partial f/\partial x_j$ is non-constant on its domain $\prod_{i=1}^{d-1}I_i$. Thus, we can rearrange indices and write $$f=px_1+qx_1x_2+rx_2+s$$ where $p,q,r,s$ are polynomials in $x_3,...,x_{d-1}$ and $q$ is not identically 0.

    Fix $\vec{a}:=(a_3,...,a_{d-1})\in \prod_{i=3}^{d-1}I_i$ such that $$q(\vec{a})\neq 0.$$ If $d=3$, then this product set is of course empty. However, in this case $p,q,r,s$ are all constant, and the argument below applies if one ignores the role of $\vec{a}$.

    Take now $$a_1\in I_1 \big\backslash \left\{\frac{-r(\vec{a})}{q(\vec{a})}\right\}, \quad a_2 \in I_2 \big\backslash \left\{\frac{-p(\vec{a})}{q(\vec{a})}\right\},$$ and set $$\alpha:= f(a_1,a_2,\vec{a}).$$ We rewrite the expression for $f$ to $$f(a_1,a_2,\vec{a})=\left(a_1+\frac{r(\vec{a})}{q(\vec{a})}\right)(x_2q(\vec{a})+p(\vec{a}))-\frac{r(\vec{a})p(\vec{a})}{q(\vec{a})}+s(\vec{a}).$$ Abbreviating $$h:= \frac{-rp}{q}+s,$$ we thus have $$a_1=\frac{\alpha - h(\vec{a})}{a_2q(\vec{a})+p(\vec{a})}-\frac{r(\vec{a})}{q(\vec{a})}.$$ 
    
    Our choices of $a_1,a_2$ mean that $h(\vec{a})\neq \alpha$, so that there exists some $0<\varepsilon<|h(\vec{a})-\alpha|$. Since all functions involved are continuous wherever they are defined, we may find open intervals $J_i\subset I_i$ for $i=2,3,...,d$ such that $\prod_{i=2}^d J_i$ is an open neighborhood of $(a_2,\vec{a},\alpha)$ and for each $$(t,\vec{b},\beta)\in \prod_{i=2}^dJ_i,$$ one has $$\frac{\beta - h(\vec{b})}{tq(\vec{b})+p(\vec{b})}-\frac{r(\vec{b})}{q(\vec{b})}\in I_1$$ and $$|\beta - \alpha|, \left|h(\vec{a})-h(\vec{b})\right|<\frac{\varepsilon}{2}.$$ 

    Allowing $t$ to vary slightly, the map $$t\mapsto\left(\frac{\beta - h(\vec{b})}{tq(\vec{b})+p(\vec{b})}-\frac{r(\vec{b})}{q(\vec{b})},t,\vec{b},\beta\right)$$ describes a $C^1$ curve in $\left((0,0,\vec{b},\beta)+E_{1,2}\right)\cap\mathcal{V}$. Moreover, this curve is non-linear since \begin{align*}
        |h(\vec{b})-\beta| &\geq |h(\vec{a})-\alpha| - |\alpha -\beta| - |h(\vec{b})-h(\vec{a})| \\
        &> |h(\vec{a})-\alpha| - \varepsilon/2 - \varepsilon/2 \\
        &> \varepsilon-\varepsilon=0
    \end{align*} 

    Setting $$\mathcal{U}:=\left\{(0,0,\vec{b},\beta): \vec{b}\in \prod_{i=3}^{d-1}J_i,\;\; \beta \in J_d\right\}$$ gives the desired subset of $E_{1,2}^\perp$ and completes the proof.

\end{proof}

\begin{proof}[Proof of Theorem \ref{thm-smooth is linear}]
    First note that when $d=1$, any smooth Markov partition is trivially linear, since 0-dimensional $C^1$ manifolds are just points. 
    
    Thus, let $d\geq 2$ and suppose $A$ is a $d\times d$ expanding integer matrix and $\M$ a smooth Markov partition for $A$. By Theorem \ref{thm-A smooth MP iff A^k diagonalizable with integer eigenvals} and Proposition \ref{prop-reduction to similarity and powers}, we may assume $$A=\begin{pmatrix}
        k_1 & & & \\
        & k_2 & & \\
        & & \ddots & \\
        & & & k_d
    \end{pmatrix}$$ where each $k_i$ is an integer and $k_1\geq k_2 \geq ... \geq k_d>1$. If $d=2$, it is immediate from Lemma \ref{lem-nonlinear curves for diagonal matrices are dense} that for any smooth Markov partition $\M$, $\partial\M$ cannot contain a non-linear $C^1$ curve, so $\M$ must be linear.

    For $d>2$, suppose towards a contradiction that $\M$ is not linear. Then in particular, there exists a non-linear hypersurface $\tilde{\mathcal{S}}\subset\partial\M$ with a lift $\mathcal{S}$ in $\R^d$. Applying Lemma \ref{lem-nonlinear surfaces have nonlinear curves}, we get an open subset $\mathcal{V}\subset\mathcal{S}$, two indices $i,j$ and an open $\mathcal{U}\subset E_{i,j}^\perp$ such that for any $w\in \mathcal{U}$, $\left(w+E_{i,j}\right)\cap \mathcal{V}$ contains a non-linear curve. By reordering the entries, we may assume $i=1,j=2$.

    We will now show that $\partial\M$ is dense in the open set $\rho(\mathcal{U}+E_{1,2})$. Let $\varepsilon>0$, $$\begin{pmatrix}
        \vec{x} \\ \vec0
    \end{pmatrix}\in E_{1,2} \text{ and } \begin{pmatrix}
        \vec0 \\ \vec{y}
    \end{pmatrix}\in \mathcal{U}.$$ By density of periodic points, there exists a $\begin{pmatrix}
        \vec{0} \\ \vec{z}
    \end{pmatrix}\in \mathcal{U}$ such that $\rho\vec{z}$ is $p$-periodic under the expanding toral endomorphism induced by $$C:=\begin{pmatrix}
        k_3 & & \\ & \ddots & \\
        & & k_d
    \end{pmatrix}$$ and $$d_\T(\vec{z},\vec{y})<\varepsilon.$$ 

    Let $$\begin{pmatrix}
        \vec{\gamma} \\ \vec{0}
    \end{pmatrix}\subset \left(\begin{pmatrix}
        \vec{0} \\ \vec{z}
    \end{pmatrix} + E_{1,2}\right)\cap\mathcal{V}$$ be a non-linear $C^1$ curve. By Lemma \ref{lem-nonlinear curves for diagonal matrices are dense}, there is an $n$ such that $$\rho \begin{pmatrix}
        k_1^p & 0 \\ 
        0 & k_2^p
    \end{pmatrix}^{n}\vec{\gamma}$$ is $\varepsilon$-dense in $\T^2$. Take $t\in[0,1]$ such that $$d_\T\left(\begin{pmatrix}
        k_1^p & 0 \\ 
        0 & k_2^p
    \end{pmatrix}^{n}\vec{\gamma}(t),\vec{x}\right)<\varepsilon.$$

    Then \begin{align*}
        d_\T\left(A^{pn}\begin{pmatrix}
            \vec{\gamma}(t) \\
            \vec{z}
        \end{pmatrix},\begin{pmatrix}
            \vec{x} \\ \vec{y}
        \end{pmatrix}\right) &\leq d_\T\left(\begin{pmatrix}
        k_1 & 0 \\ 
        0 & k_2
    \end{pmatrix}^{pn}\vec{\gamma}(t),\vec{x}\right) + d_\T\left(C^{pn}\vec{z},\vec{y}\right) \\
    &\leq \varepsilon + d_\T(\vec{z},\vec{y}) \\
    &\leq 2\varepsilon
    \end{align*} 

    Since $$\rho A^{pn}\begin{pmatrix}\vec\gamma \\ \vec z\end{pmatrix}\subset f_A^{pn}\rho(\mathcal{V})\subset \partial\M,$$ we conclude that $\partial\M$ is dense in $\rho(\mathcal{U}+E_{1,2})$, contradicting Lemma \ref{lem-boundary is invariant and nowhere dense}. Thus, $\M$ must be a linear Markov partition. 
\end{proof}

\subsection{On the hypothesis of Theorem \ref{thm-A smooth MP iff A^k diagonalizable with integer eigenvals}}

We point out that while Theorem \ref{thm-A smooth MP iff A^k diagonalizable with integer eigenvals} asks that \textit{some} power of an expanding integer matrix $A$ has integer eigenvalues, the minimal such power is actually bounded by a constant depending only on the dimension. The idea of the proof is to first consider the smallest power $\nu$ for which the eigenvalues of $A^\nu$ all have integer modulus. The characteristic polynomial of $A^\nu$ then (essentially) decomposes into cyclotomic polynomials, which allows us to bound how big $n$ needs to be for $A^n$ to have integer eigenvalues.

Let us introduce some notation. Let $\mathcal{A}_d$ be the collection of $A \in M_d(\Z)$ for which there exists an $n \in \N\setminus \{0\}$ such that $A^n$ has integer eigenvalues. The minimum such $n$ will be written $n_A$. Also, denote $$\mathcal{H}(d):=\{n_A: A\in \mathcal{A}_d\}$$
and $$\mathcal{L}(d):= \sup \mathcal{H}(d).$$ 

\begin{proposition}\label{prop:bounded_power}
    For all $d\in \N\setminus\{0\}$, $\mathcal{L}(d)$ is finite.
\end{proposition}

\begin{proof}
    Let $A\in \mathcal{A}_d$ and $\lambda_j = r_je^{i\theta_j}$ with $r_j \ge 0$ denote its eigenvalues. Observe that if the characteristic polynomial $\chi_A(x)$ is a product of two polynomials $f(x), g(x) \in \Q[x]$, then $f$ and $g$ partition the roots of $\chi_A$. Hence, $n_A$ will be the least common multiple of the smallest $n_f$ and smallest $n_g$ for which $\xi^{n_f} \in \Z$ and $\omega^{n_g} \in \Z$ for all roots $\xi$ of $f$ and $\omega$ of $g$. Since $\chi_A$ consists of at most $d$ irreducible factors, it thus suffices to bound $n_A$ by a function of $d$ when $\chi_A$ is irreducible over $\Q$.
    
    Assuming $\chi_A$ is irreducible, write $k = \lambda^{n_A} \in \Z$ for some root $\lambda$ of $\chi_A$. Then, $\lambda$ is a root of $x^{n_A} - k$, so $\chi_A$ divides $x^{n_A} - k$. This means all the roots of $\chi_A$ have the same modulus, so $r_1=r_2=...=r_d=r$ for some $r\geq0$. When $r=0$ the situation is of course trivial, so we consider $r>0$.

    Furthermore, the constant coefficient of $\chi_A$ is $\prod_{j=1}^d \lambda_j\in \Z$, so $r^d \in \N$. Thus, the smallest $\nu \in \N\setminus \{0\}$ for which $r^\nu \in \N$ is at most $d$. After raising $A$ to the power $\nu$, and taking its characteristic polynomial, we obtain 
    $$
        \chi_{A^\nu}( r^\nu x) = \prod_{j=1}^d \brac{r^\nu x - r^\nu e^{i \nu\theta_j}}= r^{\nu d} \underbrace{\prod_{j=1}^d \brac{x - e^{i \nu\theta_j}}}_{\in \Q[x]}.
    $$
    Let $\varphi$ denote Euler's totient function and $\Phi_m$ the $m^{\text{th}}$ cyclotomic polynomial, so that $\deg(\Phi_m) = \varphi(m)$. By assumption, $e^{in_A\theta_j} = \pm 1$, so $\theta_j \in \pi \Q$. That is, all the roots of $\chi_{A^\nu}( r^\nu x)$ are roots of unity. Hence, $\chi_{A^\nu}( r^\nu x)/r^{\nu d}$ is a product of cyclotomic polynomials:
    $$
        \chi_{A^\nu}( r^\nu x) = r^{\nu d} \prod_{j=1}^d \brac{x - e^{i \nu\theta_j}} = r^{\nu d} \prod_{J \in \mJ} \Phi_{m_J}(x)
    $$
    for some indexing set $\set{m_J: J \in \mJ} \subset \N$. The smallest natural $n_J$ such that $\xi^{n_J} =\pm 1$ for all roots $\xi$ of $\Phi_{m_J}$ is
    $$
        n_J = \min \Set{t \in \N: \frac{2\pi t}{m_J} = 0 \mod \pi} = \frac{m_J}{\gcd(2,m_J)} \le m_J.
    $$

    Additionally, $\deg(\Phi_{m_J}) = \varphi(m_J) \le d$ since $\deg(\chi_{A^\nu}) = d$. It is well known that $$\lim_{t \to \infty} \frac{\varphi(t)}{\sqrt{t}} = \infty,$$ see e.g. Theorem 327 of \cite{HW}. (Of course, much sharper bounds exists, but we opt for simplicity of the expression involved.) Thus, there is some constant $C>0$ so that for any $t$, $\varphi(t)\geq C\sqrt{t}$. This gives $m_J\leq C^{-2}d^2$ for all $J\in \mathcal{J}$.
    
    Thus, we conclude that
    $$
        n_A = \nu \, \lcm\set{n_J : J \in \mJ} \le d \prod_{J} n_J \le d \prod_{J} m_J \le C^{-2d}d^{2d+1} < \infty,
    $$
    as desired.
\end{proof}

To make the above argument a little more explicit, consider $d=2$ and an integer matrix $A\in \mathcal{A}_2$. If $\chi_A$ is reducible, then $n_A=1$ and we are done. Otherwise, suppose $\nu=1$, so that the eigenvalues of $A$ have modulus $k\in \N$. Then $$\chi_A(kx)=k^2\Phi(x)$$ where $\Phi$ is a cyclotomic polynomial of degree 2. There are only three such polynomials: \begin{align*}
        \Phi_3(x)&= x^2+x+1 \\
        \Phi_4(x)&=x^2+1 \\
        \Phi_6(x)&= x^2-x+1
\end{align*} However, for a primitive $m$-th root of unity $\zeta$ where $m$ is even, $\zeta^{m/2}=-1$. Thus, $n_A$ is either $2$ or $3$. 

Otherwise, $\nu=2$ and the eigenvalues of $A$ have modulus $r\in \R$ for which $r^2$ is a non-negative integer. Thus, $$\chi_{A^2}(r^2x)=r^4\Psi$$ where $\Psi$ is a product of cyclotomic polynomials of total degree 2. If $\Psi$ is reducible, then it is a product of linear factors with integer roots, so $n_A=\nu=2$. Otherwise, $\Psi$ is one of the three options above. If $\Psi=\Phi_3$, then $n_A=\nu\cdot \frac{3}{\gcd(3,2)}=6,$ if $\Psi=\Phi_4$ we get $n_A=4$, and for $\Psi=\Phi_6$ we get $n_A=6$. In other words, $$\mathcal{H}(2)\subset \{1,2,3,4,6\}.$$ 
To show equality, consider the following examples which are all expanding.

\begin{itemize}
    \item For $n_A=1$, let $$A=\begin{pmatrix}2 &0 \\ 0 & 2\end{pmatrix},$$ which has $$\chi_A=(x-2)^2.$$
    \item For $n_A=2$, let $$A=\begin{pmatrix}0 &2 \\ 1 & 0\end{pmatrix},$$ which has $$\chi_A=x^2-2=\left(x-\sqrt{2}\right)\left(x+\sqrt{2}\right).$$
    \item For $n_A=3$, let $$A=\begin{pmatrix}0 &-4 \\ 1 & 2\end{pmatrix},$$ which has $$\chi_A=x^2-2x+4=\left(x-2e^{i\pi/3}\right)\left(x-2e^{-i\pi/3}\right).$$
    \item For $n_A=4$, let $$A=\begin{pmatrix}0 &-2 \\ 1 & 2\end{pmatrix},$$ which has $$\chi_A=x^2-2x+2=\left(x-\sqrt{2}e^{i\pi/4}\right)\left(x-\sqrt{2}e^{-i\pi/4}\right).$$
    \item For $n_A=6$, let $$A=\begin{pmatrix}0 &-3\\ 1 & 3\end{pmatrix},$$ which has $$\chi_A=x^2-3x+3=\left(x-\sqrt{3}e^{i\pi/6}\right)\left(x-\sqrt{3}e^{-i\pi/6}\right).$$
\end{itemize}

In conclusion, a $2\times 2$ integer matrix $A$ is in $\mathcal{A}_2$ if and only if $A^{12}=A^{\lcm\mathcal{H}(2)}$ has integer eigenvalues.

\section{The case of dimension 2}\label{Section: Dimension 2}

For our proof of Theorem \ref{thm-A smooth MP iff A^k diagonalizable with integer eigenvals}, we relied on Mañé's result which implied that rectifiable curves have dense orbits under certain toral endomorphisms. When restricting to dimension 2, we are able to make some more subtle observations about the geometry of $\partial\M$ for certain Markov partitions. In particular, we focus on expanding matrices with complex eigenvalues that are not roots of integers and non-diagonalizable matrices.

\subsection{Complex eigenvalues}

For $A$ a $2\times 2$ expanding integer matrix with complex eigenvalues that are not roots of integers, the following lemma implies that any Markov partition for $A$ must be nowhere differentiable in the sense that its boundary has no tangent line anywhere. The proof is similar to the second case of Lemma \ref{lem-nonlinear curves for diagonal matrices are dense}. 

\begin{proposition}
    Let $A$ be a $2\times 2$ expanding integer matrix with complex eigenvalues that are not roots of integers and $\gamma$ a continuous curve such that $\gamma'(t)$ exists and is non-zero for some $t\in[0,1]$. Then for all $\varepsilon>0$ there exists an $n\in \N$ such that $\rho A^{n}\gamma$ is $\varepsilon$-dense in $\T^2$.  
\end{proposition}

\begin{proof}
    Since the eigenvalues of $A$ are complex, there are $\lambda>0,\phi\in[0,2\pi)$, and a real invertible matrix $T$ such that $$A=T^{-1}\lambda R_\phi T,$$ where $R_\phi$ is the matrix that rotates a vector by an angle $\phi$. Let $\varepsilon>0$, $\vec{v}$ an arbitrary vector with irrational slope, and let $N>0$ be such that $$\rho([0,N]\vec{v})=\rho\left(\left\{t\vec{v}:t\in[0,N]\right\}\right)$$ is $\varepsilon$-dense in $\T^2$. 

    By reparametrizing, we may assume $\gamma'(0)$ exists and is non-zero. For a non-zero vector $\vec{w}$, let $\operatorname{proj}_{\vec{w}}$ denote the orthogonal projection onto $\R\vec{w}$ and write $\theta(\vec{u},\vec{w})$ for the (unsigned) angle between two non-zero vectors $\vec{u},\vec{w}$. Observe that $$\theta(R_\phi\vec{u},R_\phi\vec{w})=\theta(\vec{u},\vec{w})$$ and for any $s_1,s_2>0$, $$\theta(s_1\vec{u},s_2\vec{w})=\theta(\vec{u},\vec{w}).$$

    Let $\delta<\pi/2$ be small enough so that $$\tan(2\delta)<\frac{\varepsilon}{N\|T^{-1}\|\|T\vec{v}\|},$$ where we use $\|\cdot\|$ for both the Euclidian norm on $\R^2$ and its induced operator norm. By definition of $\gamma'(0)$, there exists an $\eta>0$ such that for all $t\in(0,\eta]$, $\gamma(t)\neq\gamma(0)$ and $$\theta(T(\gamma(t)-\gamma(0)),T\gamma'(0))<\delta.$$ Moreover, given that the eigenvalues of $A$ are not roots of integers, $\phi$ must be irrational mod $\pi$. Therefore, there exists a sequence $n_k\to\infty$ such that for all $k$, $$\theta(R_{n_k\phi}T\gamma'(0),T\vec{v})<\delta.$$

    Taking $\delta$ small enough, we in particular have \begin{align*}
        \theta(R_{n_{k}\phi}T(\gamma(t)-\gamma(0)),T\vec{v}) &\leq \theta(R_{n_{k}\phi}T(\gamma(t)-\gamma(0)),R_{n_k\phi}T\gamma'(0)) + \theta(R_{n_k\phi}T\gamma'(0),T\vec{v}) \\
        &\leq \theta(T(\gamma(t)-\gamma(0)),T\gamma'(0))+ \delta \\
        &\leq 2\delta
    \end{align*}
    since $R_{\phi}$ preserves angles. 

    Take now $n=n_k$ large enough so that $$\lambda^{n}\|T\gamma(\eta)\|\geq \frac{N\|T\vec{v}\|}{\cos(2\delta)}.$$ Fixing an arbitrary $\vec{x}\in[0,1)^2$, there exists an $r\in[0,N]$ be such that $$d_\T(\vec{x}-A^{n}\gamma(0),r\vec{v})<\varepsilon.$$
    
    Since we are taking an orthogonal projection, \begin{align*}
        \left\|\operatorname{proj}_{T\vec{v}}\left(\lambda^{n}R_{n\phi} T(\gamma(\eta)-\gamma(0))\right)\right\|&=\cos\theta\left(R_{n\phi}T(\gamma(\eta)-\gamma(0)),T\vec{v}\right)\|\lambda^{n}T\gamma(\eta)\| \\
        &\geq \cos(2\delta)\frac{N\|T\vec{v}\|}{\cos(2\delta)}\\
        &= N\|T\vec{v}\|
    \end{align*}
    Therefore, the Intermediate Value Theorem gives us a $t\in[0,\eta]$ such that $$\operatorname{proj}_{T\vec{v}}(\lambda^{n}R_{n\phi} T(\gamma(t)-\gamma(0)))=rT\vec{v}.$$ For brevity, write $$\vec{u}=\vec{u}(n,t):=\lambda^{n}R_{n\phi} T(\gamma(t)-\gamma(0)),$$ and observe that again by orthogonality of the projection, we have \begin{align*}
        d_\R(rT\vec{v},\vec{u})=d_\R(\operatorname{proj}_{T\vec{v}}(\vec{u}),\vec{u}) &= \|\operatorname{proj}_{T\vec{v}}(\vec{u})\| \tan\theta(T\vec{v},\vec{u}).
    \end{align*}
    
    If our value of $t$ is not 0, we have \begin{align*}
        d_\R(r\vec{v},A^{n}(\gamma(t)-\gamma(0)))&= d_\R\left(T^{-1}Tr\vec{v}, T^{-1}\vec{u}\right)\\ 
        &\leq \|T^{-1}\| d_\R \left(rT\vec{v}, \vec{u}\right) \\
        &= \|T^{-1}\| \|rT\vec{v}\| \tan \theta(T\vec{v},R_{n\phi}T(\gamma(t)-\gamma(0))) \\
        &\leq \|T^{-1}\| \|T\vec{v}\| N \tan 2\delta \\
        &\leq \varepsilon
    \end{align*}
    If $t=0$, then $r=0$, so the final inequality above is obviously satisfied.

    Putting this all together yields \begin{align*}
        d_\T(\vec{x},A^n\gamma(t))&\leq d_\T(\vec{x}-A^n\gamma(0),r\vec{v})+d_\R(r\vec{v},A^{n}(\gamma(t)-\gamma(0))) \\
        &\leq 2\varepsilon
    \end{align*}
    completing the proof. 
\end{proof}

\begin{figure}
    \centering
    \includegraphics[width=0.5\linewidth]{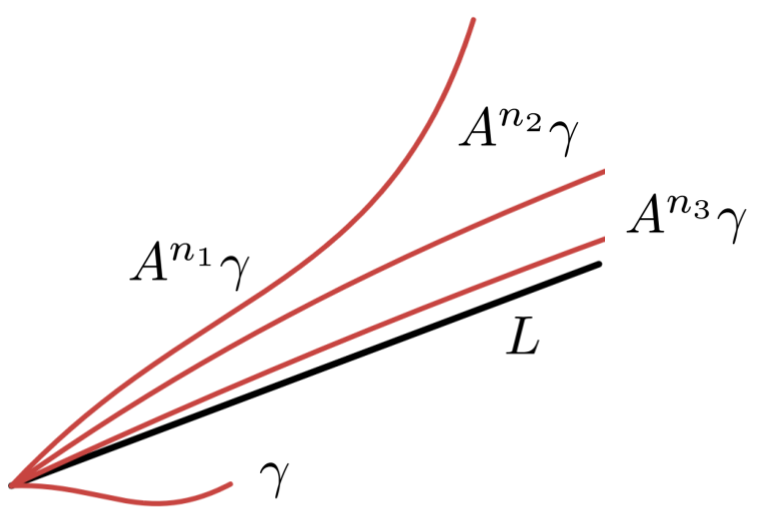}
    \caption{Iterates of $\gamma$ under a matrix $A$ with complex eigenvalues whose arguments do not lie in $\pi\Q $ where $\gamma(0)=0$. $L=\R^+T\vec{v}$ and $n_k$ are chosen so that $A^{n_k} \gamma^\prime(0)$ approaches the direction of $L$ counterclockwise.}
    \label{fig:complex_eigenvalues}
\end{figure}

\subsection{Jordan blocks}

To handle non-diagonalizable matrices, we used Proposition \ref{prop-Jordan block}, which applies to curves with some vertical change. This is necessary because, for example, a line segment along the $x$-axis will not have dense orbit under a Jordan matrix $\begin{pmatrix}
    k & 1 \\ 0 & k
\end{pmatrix}$. In fact, for a $2\times 2$ expanding integer matrix that is not diagonalizable, it is possible to construct a Markov partition which contains linear segments. Such a Markov partition is interesting because it exhibits a ``hybrid'' type of smoothness.

\begin{observation}
    Let $A$ be a $2\times 2$ non-diagonalizable expanding integer matrix. Then there exists a Markov partition $\M$ for $A$ for which $\partial\M$ contains a linear line segment. 
\end{observation}

This follows from using Proposition \ref{prop-reduction to similarity and powers} to reduce to the case that $A$ is a Jordan matrix with positive integer eigenvalues, and applying the construction of Bedford \cite{Be}. We sketch the method briefly, but refer to \cite{Be} and \cite{D} for details.

For a given expanding integer matrix $A$, consider the images of the unit vectors $Ae_1$ and $Ae_2$. Each of these vectors is approximated by a sequence of elements in $\Z^2$, which represent a continuous path consisting of vertical and horizontal unit line segments (with some additional non-triviality constraints). We encode this path as a word on the symbols $s_1,s_2,s_1^{-1},s_2^{-1}$, where $s_1$ corresponds to a positive step in the horizontal direction, $s_2$ to a positive step in the vertical direction, and the inverses to negative increments. 

For example, when $$A=\begin{pmatrix}
    2 & 1 \\ 0 & 2
\end{pmatrix},$$ the path for $Ae_1$ is $s_1s_1$, and the path for $Ae_2$ is $s_2s_2s_1$. Since there is a natural correspondence between $e_i$ and $s_i$, we obtain a map $\theta$ from the free semigroup generated by $\{s_1,s_2,s_1^{-1},s_2^{-1}\}$ to itself. We work with the free semigroup rather than the free group for technical reasons. Sticking with our example matrix $A$, we get $$\theta(s_1)=s_1s_1, \quad \theta(s_2)=s_2s_2s_1.$$ We are now able to iterate this map, $$\theta^2(s_1)=s_1s_1s_1s_1, \quad \theta^2(s_2)=s_2s_2s_1s_2s_2s_1s_1s_1,$$ or visually

$$\begin{circuitikz}
    \draw [->, >=Stealth] (0,0) -- (1,0);

    \draw [->, >=Stealth, dashed] (1.5,0) .. controls (2,0.5) and (2.5,0.5) .. (3,0)node[pos=0.5, fill=white]{$\theta$};

    \draw [->, >=Stealth] (3.5,0) -- (4.5,0);
    \draw [->, >=Stealth] (4.5,0) -- (5.5,0);

    \draw [->, >=Stealth, dashed] (6,0) .. controls (6.5,0.5) and (7,0.5) .. (7.5,0)node[pos=0.5, fill=white]{$\theta$};

    \draw [->, >=Stealth] (8,0) -- (9,0);
    \draw [->, >=Stealth] (9,0) -- (10,0);
    \draw [->, >=Stealth] (10,0) -- (11,0);
    \draw [->, >=Stealth] (11,0) -- (12,0);

    \draw [->, >=Stealth] (0.5,-4.5) -- (0.5,-3.5);
    
    \draw [->, >=Stealth, dashed] (1.5,-4) .. controls (2,0.5-4) and (2.5,0.5-4) .. (3,-4)node[pos=0.5, fill=white]{$\theta$};

    \draw [->,>=Stealth] (4,-5) -- (4,-4);
    \draw [->,>=Stealth] (4,-4) -- (4,-3);
    \draw [->,>=Stealth] (4,-3) -- (5,-3);

    \draw [->, >=Stealth, dashed] (6,-4) .. controls (6.5,0.5-4) and (7,0.5-4) .. (7.5,0-4)node[pos=0.5, fill=white]{$\theta$};

    \draw [->, >=Stealth] (8,-6) -- (8,-5);
    \draw [->, >=Stealth] (8,-5) -- (8,-4);
    \draw [->, >=Stealth] (8,-4) -- (9,-4);
    \draw [->, >=Stealth] (9,-4) -- (9,-3);
    \draw [->, >=Stealth] (9,-3) -- (9,-2);
    \draw [->, >=Stealth] (9,-2) -- (10,-2);
    \draw [->, >=Stealth] (10,-2) -- (11,-2);
    \draw [->, >=Stealth] (11,-2) -- (12,-2);
    
\end{circuitikz}$$

To construct the Markov partition, we take the unit square encoded as $s_1s_2s_1^{-1}s_2^{-1}$, apply $\theta$ $n$ times, and rescale the corresponding path in $\R^2$ by $A^{-n}$. The limiting path as $n$ goes to infinity is the boundary of a set that tiles $\R^2$, and this tiling gives us the Markov partition (after rescaling). 

When $$A=\begin{pmatrix}
    k & 1 \\ 0 & k
\end{pmatrix}$$ with $k$ a positive integer, $\theta^n(s_1)$ corresponds to a straight line from the origin to $(k^n,0)$. Of course, rescaling such a path by $A^{-n}$ gives a unit line segment along the horizontal axis. That is to say, the Markov partition obtained in this way contains a line segment, and we have already shown that it cannot be entirely smooth.

Figure \ref{fig:hybrid MP} shows an approximation of the Markov partition obtained in this way with $k=2$. The Markov partition consists of four elements, all geometrically congruent to each other. Moreover, the image of each element under the expanding endomorphism is all of $\T^2$.

\begin{figure}
    
    \centering
        \resizebox{0.5\linewidth}{!}{\input{MP_picture}}
        
    \caption{A Markov partition for $\begin{pmatrix}
        2 & 1 \\ 0 & 2
    \end{pmatrix}$. The dashed green lines indicate a fundamental domain for $\T^2$.}
    \label{fig:hybrid MP}
\end{figure}

\section{Hausdorff dimension}\label{Section: Hausdorff dim}

In this section, we estimate the Hausdorff dimension $\dim_H(\partial\M)$, using techniques from symbolic dynamics. We remark that a result of Urbánski \cite{U} can be applied to bound the box dimension, or capacity, of $\partial\M$ from below under some assumptions on the matrix inducing the endomorphism. For the explicit construction of Bedford \cite{Be}, some bounds from Dekking apply \cite{D}. Moreover, a recent result due to Ishii and Oka \cite{IO} gives a precise formula for $\dim_H(\partial\M)$, under a number of assumptions on the matrix inducing the endomorphism and the particular construction of $\M$. The bound given in Theorem \ref{thm-dimH} strongly resembles their formula when $A$ is conformal. The theorem is also in the spirit of the relation between Hausdorff dimension, entropy, and Lyapunov exponents developed by Manning and Young \cite{M1,Y}.

Let us briefly recall the definition of the Hausdorff dimension of a set, for details we refer to e.g. \cite{F}. Let $(X,d)$ be a metric space and $F\subset X$. For $\delta>0$, and $s\in [0,\infty)$ define $$H_\delta^s(F):= \inf\left\{\sum_{i=0}^\infty\diam(U_i)^s:\diam(U_i)\leq \delta \text{ and } F\subset \bigcup_{i=0}^\infty U_i\right\}$$ where the infimum ranges over all countable collections of subsets of $X$ that cover $F$. Next define $$H^s(F)=\lim_{\delta\to 0} H_\delta^s(F).$$ The Hausdorff dimension of $F$ is then $$\dim_H(F):= \inf \{s\geq 0 : H^s(F)=0\}=\sup \{s\geq 0:H^s(F)=\infty\},$$ with $\sup\emptyset = 0$.  

The first part of the argument here will be given in more generality. Recall from Section \ref{Section:Prelims} that for a positively expansive map $f:X\to X$ from a manifold onto itself, \cite{CR} shows we can find a compatible metric and  $\lambda>1,\varepsilon>0$ for which \begin{align*}d(x,y)<\varepsilon \implies d(fx,fy)\geq \lambda d(x,y).\tag{*}\end{align*} For the moment, we will work with a metric that is \textit{adapted} in this sense. 

Suppose we are given an expanding dynamical system $f:X\to X$ where $X$ is equipped with an adapted metric and $\M$ is a Markov partition for $f$. Any expanding dynamical system is positively expansive, with expansive constant $c=\varepsilon$. We proved Lemma \ref{lem-backwards refinements of MP} for expanding toral endomorphisms, but the proof holds for any surjective, expanding map, since such a map is locally a homeomorphism. By the lemma, we may thus assume that for each rectangle $R\in \M$ we have $\operatorname{diam}(R) < c/2$. Let $T$ be the corresponding transition matrix and $\pi:\Sigma_T\to X$ be the SFT and covering map coming from $T$. Then, for any $x,y\in \Sigma_T$, we have that $$\pi(x)=\pi(y) \iff \forall i\in \N, R_{x_i}\cap R_{y_i}\neq \emptyset.$$

Mirroring a classical construction of Manning \cite{M2}, define a graph with vertex set $$\{(a,b): 1\leq a,b \leq |\M|, a\neq b, R_a\cap R_b \neq \emptyset\}$$ and an edge $(a,b)\to (c,d)$ if and only if $T_{ac}=T_{bd}=1$. Call the shift on this new graph $E_2$. We now define a map $\pi_2:E_2\to X$ by $$\pi_2((x_0,y_0),(x_1,y_1),...) = \pi(x_0x_1...).$$

It follows from Theorem 1.5 of \cite{AKS} that $\pi_2$ is surjective onto $\partial\M$ and for some $M$, $$\sup_{x\in \partial\M} |\pi_2^{-1}(x)|\leq M.$$ As a consequence, $\pi_2$ preserves topological entropy $h(\cdot)$, i.e. $$h(E_2)=h(f|_{\partial\M}).$$ For a purely topological proof of this fact, see Chapter IX, Theorem 1.8 of \cite{Ro}.

\begin{theorem}\label{thm-dimH}
    In the setup as above, $$\dim_H(\partial\M)\leq\frac{h(f|_{\partial\M})}{\log \lambda}.$$ 
\end{theorem}

\begin{proof} 

For a shift space $\Sigma$, write $B_n(\Sigma)$ for the \textit{allowed words} of length $n$ in $\Sigma$, that is, all sequences of length $n$ that appear in some point in $\Sigma$. Given a word $w$, let $C_w$ denote the cylinder set corresponding to the word in the appropriate subshift. Moreover, for a word $w$ allowed in $E_2$, let $w^1$ be the word formed by the first coordinates.

First, for any $n\in \N$, 
$$\partial \M = \bigcup_{w\in B_n(E_2)} \pi_2(C_w) \subset \bigcup_{w\in B_n(E_2)} \pi(C_{w^1}).$$ We only need to consider words from $E_2$ since $f(\partial \M)\subset \partial \M$. Note that $C_w\subset E_2$, while $C_{w^1}$ is a subset of $\Sigma_T$. 

By \cite{AKS} Lemma 3.2, there is some $D\in \R$ so that for any sufficiently long word $w'\in B_n(\Sigma_T)$, $$\operatorname{diam}(\pi(C_{w'})) \leq D \frac{1}{\lambda^n}.$$ 

Let $\mu\in \R$ be such that $h(E_2)=\log\mu$ and fix $\varepsilon>0$. Then for all $n$ sufficiently large, $\# B_n(E_2) \leq e^{n\varepsilon}\mu^n$ by theorem 4.4.4 of \cite{LM}.

Hence, for fixed $s\in \R^+$ and $n$ large enough, we have that \begin{align*}
\sum_{w\in B_n(E_2)} \operatorname{diam}(\pi_2(C_w))^s &\leq \sum_{w\in B_n(E_2)}\operatorname{diam}(\pi(C_{w^1}))^s \\
&\leq\sum_{w\in B_n(E_2)}\left(D \frac{1}{\lambda^n}\right)^s  \\
&\leq \#B_n(E_2) \left(D \frac{1}{\lambda^n}\right)^s \\
&= D^s\frac{e^{n\varepsilon}\mu^n}{\lambda^{ns}}
\end{align*}

As $n\to \infty$, this quantity goes to zero whenever $s>\frac{\varepsilon+\log\mu}{\log\lambda}$. Since our choice of $\varepsilon$ was arbitrary, we conclude that $$\dim_H(\partial\M)\leq \frac{\log\mu}{\log\lambda} = \frac{h(E_2)}{\log\lambda}=\frac{h(f|_{\partial\M})}{\log\lambda},$$ as desired.

\end{proof}

Focusing again on expanding toral endomorphisms, we are interested in the Hausdorff dimension of $\partial\M$ with respect to the standard metric. The metric given by Coven and Reddy is bi-Lipschitz equivalent to the standard metric, so Hausdorff dimensions are preserved. However, the metric is constructed in full generality. In the case of expanding toral endomorphisms, we can express the expansion factor $\lambda$ in terms of the matrix inducing the endomorphism, making our bound on the Hausdorff dimension more explicit. 
 
\begin{corollary}
    Let $f=f_A:\T^d\to \T^d$ be an expanding toral endomorphism with Markov partition $\M$ and suppose $\T^d$ is equipped with the standard metric. Then $$\dim_H(\partial\M) \leq \frac{h(f|_{\partial\M})}{\log \lambda}$$ where $$\lambda=\min\{|\mu|:\mu \text{ is an eigenvalue of } A\}.$$
\end{corollary}

\begin{proof}
    Recall that $$\min_{\|x\|=1}\|Ax\|=\frac{1}{\|A^{-1}\|},$$ where $\|\cdot\|$ denotes both the standard Euclidian norm on $\R^d$ and the induced operator norm. With $\lambda$ as in the statement, Gelfand's formula (Chapter 17, Theorem 4 of \cite{La}) asserts that \begin{align*}
        \limn \|A^{-n}\|^\frac{1}{n} &= \lambda^{-1}.
    \end{align*}
    Fix $\varepsilon>0$, and let $n\in\N$ such that $$\|A^{-n}\|\leq (1+\varepsilon)^n\lambda^{-n}$$ then $$\min_{\|x\|=1}\|A^nx\|\geq (1+\varepsilon)^{-n}\lambda^n.$$ Whenever $d_\T(x,y)<(2\|A^n\|)^{-1},$ $f_{A^n}x$ and $f_{A^n}y$ are contained in a ball where $d_\T$ and $d_\R$ agree, so since $f_{A^n}=f^n$, we see that $$d_\T(f^nx,f^ny) \geq (1+\varepsilon)^{-n}\lambda^nd_\T(x,y).$$ 
    
    We now use the same metric as given in Lemma 2 of \cite{CR}. Define $$\hat{d}(x,y):= \sum_{i=0}^{n-1}\mu^{-i}d_\T(f^ix,f^iy)$$ where $\mu:=(1+\varepsilon)^{-1}\lambda$. Then \begin{align*}
        \hat{d}(fx,fy)&=d_\T(fx,fy)+\mu^{-1}d_\T(f^2x,f^2y)+... + \,\mu^{-(n-1)}d_\T(f^nx,f^ny)  \\
        &\geq d_\T(fx,fy)+\mu^{-1}d_\T(f^2x,f^2y) + ... + \mu d_\T(x,y) \\
        &= \mu \hat{d}(x,y)
    \end{align*}
    whenever $d_\T(x,y)\leq \hat{d}(x,y)< (2\|A^n\|)^{-1}$. The metrics $\hat{d}$ and $d_\T$ are strongly equivalent, so the topological entropy of $f|_{\partial\M}$ is the same when computed with either metric, see Theorem 7.4 in \cite{W}. The same is true for Hausdorff dimension, so using Theorem \ref{thm-dimH}, we get $$\dim_H(\partial\M)\leq \frac{h(f|_{\partial\M})}{\log \lambda-\log(1+\varepsilon)}.$$ Taking $\varepsilon$ to 0 completes the proof. 
\end{proof}


\newpage

\begin{thebibliography}{WWWWW}

    \bibitem[AKS]{AKS} J. Ashley, B. Kitchens, and M. Stafford, {\em Boundaries of Markov partitions}, Trans. Amer. Math. Soc. {\bf 333}, 1 (1992), 177-201.

    \bibitem[AW]{AW} R. L. Adler and B. Weiss, {\em Similarity of automorphisms of the torus}, Mem. Amer. Math, Soc.,(1970).

    \bibitem[Be]{Be} T. Bedford, {\em Generating special Markov partitions using fractals}, Ergod. Th.  Dynam. Sys. {\bf 6} (1986), 325-333. 

    \bibitem[Boo]{Boo} W. M. Boothby, {\em An Introduction to Differentiable Manifolds and Riemannian Geometry}, Pure and Appl. Math. {\bf 63} (1975).

    \bibitem[Bou]{Bou} N. Bourbaki, {\em Topologie Générale: Chapitres 5 à 10}, Élém. Math. (2007). 
    
    \bibitem[Bo1]{Bo1} R. Bowen, {\em Equilibrium States and the Ergodic Theory of Anosov Diffeomorphisms}, 2nd ed., Lect. Notes Math. {\bf 470} (2008). 

    \bibitem[Bo2]{Bo2} R. Bowen, {\em Markov partitions are not smooth}, Proc. Amer. Math. Soc. {\bf 71} (1978), 130-132.

    \bibitem[Bo3]{Bo3} R. Bowen, {\em Markov Partitions and Minimal Sets for Axiom A Diffeomorphisms}, Amer. Jour. Math. (1970), 907-918.
    
    \bibitem[C]{C} E. Cawley, {\em Smooth Markov partitions and toral automorphisms}, Ergod. Th. Dynam. Syst. {\bf 11} (1991), 633-651.

    \bibitem[CR]{CR} E. Coven and W. Reddy, {\em Positively expansive maps of compact manifolds}.  In: Global Theory of Dynamical Systems, Ed: Nitecki, Z., Robinson, C.. Lect. Notes in Math. {\bf 819} (1980). 

    \bibitem[D]{D} F. Dekking, {\em Recurrent sets}, Advances in Math. {\bf 44} (1982), 78-104.

    \bibitem[F]{F} K. Falconer, {\em Techniques in Fractal Geometry}, (1997). 


    \bibitem[HW]{HW} G. Hardy and E. Wright, {\em An Introduction to the Theory of Numbers}, 6th ed. (2008). 
    
    \bibitem[IO]{IO} Y. Ishii and T. Oka, {\em On the Hausdorff dimension of the recurrent sets induced from endomorphisms of free groups},
    J. Fractal Geom. {\bf 9} (2022), 171–192. 

    \bibitem[La]{La} P. D. Lax, {\em Functional Analysis}, (2002).

    \bibitem[L]{L} C. Leuridan, {\em Expanding toral endomorphisms are Bernoulli: towards explicit isomorphisms}, to appear in: Monatsh. für Math.

    \bibitem[LM]{LM} D. Lind and B. Marcus, {\em Introduction to Symbolic Dynamics and Coding}, 2nd ed. (2021).

    \bibitem[LN]{LN} J. Llibre and N. Neumärker, {\em Period sets of toral endomoprhisms on $\T^2$}, Topo. and its Appl. {\bf 185-186} (2015), 41-49.

    \bibitem[Ma]{Ma} R. Mañé, {\em Orbits of paths under hyperbolic toral automorphisms}, Proc. Amer. Math. Soc. {\bf 73}, 1 (1979), 121-125.

    \bibitem[M1]{M1} A. Manning, {\em A relation between Lyapunov exponents, Hausdorff dimension and entropy}, Ergod. Th. Dynam. Sys. {\bf 1}, 4 (1981), 451-459.

    \bibitem[M2]{M2} A. Manning, {\em Axiom A diffeomorphisms have rational zeta functions}, Bull. Lon. Math. Soc. {\bf 3}, 2 (1971), 215-220.

    \bibitem[Mi]{Mi} E. Mihailescu, {\em Higher dimensional expanding maps and toral extensions}, Proc. Amer. Math. Soc. {\bf 141}, 2 (2013), 3467-3475.


    \bibitem[N]{N} M. Newman, {\em Integral Matrices}, Pure and Appl. Math. {\bf 45} (1972).  

    \bibitem[Ro]{Ro} C. Robinson. {\em Dynamical Systems: Stability, Symbolic Dynamics and Chaos}, (1995).
   
    \bibitem[Ru]{Ru} D. Ruelle. {\em Thermodynamic Formalism}, 2nd ed. (2004).
    
    \bibitem[Si]{Si} Ya. G. Sinai, {\em Construction of markov partitions},  Func. Ana. App. {\bf 2}, 3 (1968), 245-253.

    \bibitem[U]{U} M. Urbánski, {\em On the capacity of a continuum with non-dense orbit under a hyperbolic toral automorphism}, Studia Math. {\bf 81} (1985), 37-51.

    \bibitem[URM]{URM} M. Urbánski, M. Roy, and S. Munday, {\em Non-Invertible Dynamical Systems. Volume 1: Ergodic Theory – Finite and Infinite, Thermodynamic Formalism, Symbolic Dynamics and Distance Expanding Maps}, (2022). 

    \bibitem[Y]{Y} L.-S. Young, {\em Dimension, entropy and Lyapunov exponents}, Ergod. Th. Dynam. Sys. {\bf 2}, 1 (1982), 109-124.

    \bibitem[W]{W} P. Walters, {\em An Introduction to Ergodic Theory}, (1982). 
	
\end{thebibliography}
\end{document}